\documentclass[a4paper,11pt,reqno]{amsart}

\usepackage[utf8]{inputenc}
\usepackage[english]{babel}
\usepackage{amsmath}
\usepackage{amssymb}
\usepackage{amsfonts}
\usepackage{amsthm}
\usepackage{bbm}
\usepackage{enumitem}
\usepackage{xcolor}
\usepackage[colorlinks=false]{hyperref}
\newcommand{\wto}{\rightharpoonup}
\newcommand{\eps}{\varepsilon}
\newcommand{\epsj}{{\varepsilon_j}}
\newcommand{\mc}{\mathcal}
\newcommand{\xeps}{{u^\eps}}
\newcommand{\xepsd}{{\dot{u}^\eps}}
\newcommand{\xepsdd}{{\ddot{u}^\eps}}
\newcommand{\xepsj}{{u^\epsj}}

\newcommand{\RR}{\mathcal{R}}

\newcommand{\M}{\mathbb{M}}
\newcommand{\V}{\mathbb{V}}
\newcommand{\R}{\mathbb{R}}
\newcommand{\N}{\mathbb{N}}
\newcommand{\Q}{\mathbb{Q}}
\newcommand{\U}{\mathbb{U}}
\renewcommand{\to}{\rightarrow}
\renewcommand{\d}{\,\mathrm{d}}
      
\newcommand{\norm}[1]{\left\lVert#1\right\rVert}  
\newcommand{\scal}[2]{\left\langle #1,#2\right\rangle}
\DeclareMathOperator*{\essup}{ess\,sup}
\DeclareMathOperator{\dist}{dist}

\numberwithin{equation}{section}
\newtheorem{thm}{Theorem}[section]
\newtheorem{defi}[thm]{Definition}
\newtheorem{prop}[thm]{Proposition}
\newtheorem{lemma}[thm]{Lemma}
\newtheorem{cor}[thm]{Corollary}

\theoremstyle{definition}
\newtheorem{rmk}[thm]{Remark}

\begin{document}
	
	\author[F. Riva, G. Scilla and F. Solombrino]{Filippo Riva, Giovanni Scilla and Francesco Solombrino}
	
	\title [Inertial Balanced Viscosity and Inertial Virtual Viscosity solutions]{The notions of Inertial Balanced Viscosity and Inertial Virtual Viscosity solution for rate-independent systems}
	
	\begin{abstract}
		{The notion of Inertial Balanced Viscosity (IBV) solution to rate-independent evolutionary processes is introduced. Such solutions are characterized by an energy balance where a suitable, rate-dependent, dissipation cost is optimized at jump times. The cost is reminiscent of the limit effect of small inertial terms. Therefore, this notion proves to be a suitable one to describe the asymptotic behavior of evolutions of mechanical systems with rate-independent dissipation in the limit of vanishing inertia and viscosity. It is indeed proved, in finite dimension, that these evolutions converge to IBV solutions. If the viscosity operator is neglected, or has a nontrivial kernel, the weaker notion of Inertial Virtual Viscosity (IVV) solutions is introduced, and the analogous convergence result holds. Again in a finite-dimensional context, it is also shown that IBV and IVV solutions can be obtained via a natural extension of the Minimizing Movements algorithm, where the limit effect of inertial terms is taken into account.}
	\end{abstract}
	
	\maketitle
	
	{\small
		\par
	}

	\pagenumbering{arabic}
	
	\medskip
	
	\tableofcontents
	
	\section{Introduction}

	Rate-independent evolutions frequently occur in physics and mechanics when the problem under consideration presents such small rate-dependent effects, as inertia or viscosity, that can be neglected. Several applications can be (formally) modelled by the doubly nonlinear differential inclusion
	\begin{equation}
		\begin{cases}
			\partial \mathcal{R}(\dot{u}(t)) + D_x\mathcal{E}(t,u(t))\ni0\,,&\text{in $X^*$,}\quad \text{for a.e. }t\in [0,T],\\
			u(0)=u_0\,,
		\end{cases}
		\label{quasistprob}
	\end{equation}
	where $\mc R$ is a rate-independent dissipation potential, while $\mc E$ is a time-dependent potential energy. In this paper we limit ourselves to the case of a finite-dimensional normed space $X$, although we plan to extend the whole analysis to the infinite-dimensional context, where several difficulties concerning weak topologies and nonsmoothness of $\mc E$ naturally arise (we refer to the monograph \cite{MielkRoubbook} for more details, also compare \cite{MielkRosSav12} with \cite{MielkRosSav16}).\par 
 If the driving energy $\mc E$ is nonconvex, continuous solutions to \eqref{quasistprob} are not expected to exist, and thus in the past decades huge efforts have been spent to develop weak notions of solution capable of describing the behaviour of the system at jumps. A first attempt can be found in the notion of \emph{Energetic solution} \cite{MielkTheil1, MielkTheil} based on a global stability condition together with an energy balance which must hold for every $t\in[0,T]$: 
 \begin{equation*}
 	\begin{cases}\displaystyle
 		{\rm (GS)}\quad\mc E(t,u(t))\le \mc E(t,x)+\mc R(x-u(t)),&\text{for every }x\in X,\\
 		\displaystyle	{\rm (EB)}\quad\mc E(t,u(t))+V_\mc R(u;0,t)=\mc E(0,u_0)+\int_{0}^{t}\partial_t\mc E(r,u(r))\d r.&
 	\end{cases}
 \end{equation*}
	Condition (GS) actually turns out to be still too restrictive in the nonconvex case, where a local minimality condition would be preferable. Starting from this consideration, in \cite{MielkRosSav16,MielkRosSav12,MielkRosSav09} the notion of \emph{Balanced Viscosity solution} has been introduced and analysed. See also the recent paper \cite{Rind}. Their idea relies on the fact that physical solutions to \eqref{quasistprob} should arise as the vanishing-viscosity limit of a richer and more natural viscous problem
	\begin{equation}\label{viscous}
		\begin{cases}
			\eps \mathbb{V}\xepsd(t)+\partial\mc R(\xepsd(t))+D_x\mc E(t,\xeps(t))\ni 0,&\text{in $X^*$,}\quad \text{for a.e. }t\in [0,T],\\
			\xeps(0)=u^\eps_0\,,
		\end{cases}
	\end{equation}
	as the parameter $\eps\to 0$. Here, $\V$ denotes a symmetric positive-definite linear operator modelling viscosity. Actually, in \cite{MielkRosSav16, MielkRosSav12} more general viscous potentials are considered. The resulting evolution, called Balanced Viscosity (BV) solution, turns out to satisfy a local stability condition together with an augmented energy balance:	 
	\begin{equation*}
		\begin{cases}\displaystyle
			{\rm (LS)}\quad -D_x\mc E(t,u(t))\in \partial \mc R (0),&\text{for a.e. }t\in [0,T],\\
			{\rm (EB^{\V})}	\displaystyle\quad\mc E(t,u(t))+V^\V_\mc R(u;0,t)=\mc E(0,u_0)+\int_{0}^{t}\partial_t\mc E(r,u(r))\d r,& \text{for all }t\in [0,T].
		\end{cases}
	\end{equation*}
	{While} in (EB) the classical total variation (actually, $\mc R$-variation, see Definition~\ref{rvar}) controls both the continuous part {$u_{\rm co}$} of the evolution and the jump part, {as} it holds (see also \eqref{contvar})
	\begin{equation*}
		V_\mc R(u;0,t)=V_\mc R(u_{\rm co};0,t)+ \sum_{r\in J_u\cap[0,t]}  \left(\mc R(u(r){-}u^-(r))+\mc R(u^+(r){-}u(r))\right),
	\end{equation*}
	in (EB$^\V$) the jump part of the \lq\lq viscous\rq\rq variation involves a more complicated cost function (a Finsler distance) which takes into account the original presence of viscosity:
	\begin{equation*}
		V^\V_\mc R(u;0,t)=V_\mc R(u_{\rm co};0,t)+\sum_{r\in J_u\cap[0,t]}\left(  c^{\V}_r(u^-(r),u(r))+c^{\V}_r(u(r),u^+(r))\right).
	\end{equation*}
	At time $t\in[0,T]$, the viscous cost function is obtained as
	\begin{equation}\label{costvisc}
		c^\V_t(u_1,u_2)\!:=\!\inf\left\{ \!\int_{0}^{1}\!\!\!\!p_\V(\dot{{v}}(r),{-}D_x\mc E(t,v(r)))\d r\!\mid\! v\in W^{1,\infty}(0,1;X), v(0)=u_1,v(1)=u_2\right\}\!,
	\end{equation}
	where $p_\V$, called \emph{vanishing-viscosity contact potential}, is a suitable density arising from both the viscous and the rate-independent dissipation (see also Definition~\ref{contpot}).\par 
	In \cite{MielkRosSav16, MielkRosSav12} it is also shown that BV solutions can be obtained as the limit of time discrete approximations, solutions to the recursive discrete-in-time variational incremental scheme with time step $\tau$
		\begin{equation}\label{schemevisc}
		u^k_{\tau,\varepsilon}\in \mathop{\rm arg\,min}\limits_{x\in X}\left\{ \frac{\eps}{2\tau}\|x-u_{\tau,\varepsilon}^{k-1}\|^2_{\mathbb{V}}+\mc R\left({x-u_{\tau,\varepsilon}^{k-1}}\right) + \mc E(t^{k},x)\right\}, \quad k=1,\dots,T/\tau,
	\end{equation}
	when sending simultaneously $\eps$ and $\tau$ to $0$ with also $\tau/\eps\to 0$.\par 
	 Although the notion of BV solution turned out to be extremely powerful in applications (see for instance \cite{AlLazLuc, CrRossi,KneRosZan, MielkRosSavmulti}), it still lacks inertial terms, which are however essential in the description of real world phenomena, as stated by the second principle of dynamics.\par
	 In this paper we thus present a novel notion of solution which takes into account this feature. Our starting point is augmenting \eqref{viscous} with an inertial term	 
	\begin{equation}
		\begin{cases}
			\eps^2 \mathbb{M}\xepsdd(t)+\eps \mathbb{V}\xepsd(t)+\partial\mc R(\xepsd(t))+D_x\mc E(t,\xeps(t))\ni 0, \quad\text{in $X^*$,}\quad \text{for a.e. }t\in [0,T],\\
			\xeps(0)=u^\eps_0,\quad\xepsd(0)=u^\eps_1,
		\end{cases}
		\label{mainprob}
	\end{equation}
	and then sending $\eps\to 0$, namely performing a \emph{vanishing-inertia and viscosity} argument. The symmetric positive-definite linear operator $\M$ appearing in \eqref{mainprob} represents masses. Its presence allows to consider also the case of null viscosity, i.e. $\V=0$, or more generally of a positive-semidefinite linear operator. We point out that such a limit procedure is also known as \emph{slow-loading} limit, since \eqref{mainprob} comes from a dynamic problem with slow data after a reparametrization of time. We refer the interested reader to \cite{GidRiv} or to \cite[Chapter 5]{MielkRoubbook} for a more detailed explanation.\par 
	This approach has already been adopted for concrete models (in infinite dimension) in \cite{DMSap, DMSca,LazNar, LazRosThomToad,MielkPetr,Rivquas,Sca}, all in the case of convex, or even quadratic, energies $\mc E$. An abstract analysis has been performed in \cite{GidRiv}, in finite dimension and always under convexity assumptions. {Hence, the results contained in this paper on the one hand represent an extension to nonconvex energies of the ones presented in \cite{GidRiv} (see Remark~\ref{rmk:ex}), and on the other hand put the basis for an abstract investigation in infinite dimension, where nonconvex problems are common in applications (we refer again to \cite{AlLazLuc, CrRossi,KneRosZan, MielkRosSavmulti, MielkRoubbook}).}\par
 In our nonconvex setting, the limiting evolution of \eqref{mainprob} provides the novel notion of \emph{Inertial Balanced Viscosity (IBV) solution} (we refer to the discussion in Section~\ref{sec:IBV} for the subtler notion of \emph{Inertial Virtual Viscosity (IVV) solution}, arising when the viscosity operator $\V$ is not positive-definite) for the rate-independent system \eqref{quasistprob}, namely a function satisfying the same local stability condition (LS) than BV solutions, together with an energy balance in which the cost at jump points is sensitive of the presence of inertia:
		\begin{equation*}
		\begin{cases}\displaystyle
			{\rm (LS)}\quad\,\,\, -D_x\mc E(t,u(t))\in \partial \mc R (0),\quad\quad\quad\quad\text{for a.e. }t\in[0,T]\\
			{\rm (EB^{\M,\V})}	\displaystyle\,\,	\mc E(t,u^+(t)){+}V_\RR(u_{\rm co};s,t)+\!\!\!\!\!\!\!\sum_{r\in J^{\rm e}_u\cap[s,t]}\!\!\!\!\!\! c^{\M,\V}_r(u^-(r),u^+(r))=	\mc E(s,u^-(s)){+}\!\!\int_{s}^{t}\!\!\!\partial_t\mc E(r,u(r))\d r,\\
			\quad\quad\quad\quad\text{for every }0\le s\le t\le T,
		\end{cases}
	\end{equation*}
	where now the cost turns out to be 
		\begin{equation*}
		c^{\mathbb{M},\V}_t(u_1,u_2):=\inf\left\{\int_{-N}^{N} p_\V(\dot{v}(r),-\mathbb{M}\ddot{v}(r)-D_x\mc E(t,v(r)))
		\,\mathrm{d}r\mid \,\, N\in\N,\, v\in V^{\mathbb{M},N}_{u_1,u_2} \right\}\,,
	\end{equation*}
	where
	\begin{equation*}
		V^{\mathbb{M},N}_{u_1,u_2}\!:=\!\left\{\!v\in W^{2,\infty}({-}N,N;X)\mid v({-}N)\!=\!u_1, v(N)\!=\!u_2,\, \M\dot{v}(\pm N)\!=\!0,\, \essup_{r\in [-N,N]}\|\M\ddot{v}(r)\|_* \!\le\! \overline{C}\!\right\}.
	\end{equation*}
	Notice that, differently than in the vanishing-viscosity case, the dissipation cost we consider is no longer invariant under a time-reparametrization, due to the presence of the second-order term $\mathbb{M}\ddot{v}(r)$ inside the integral. This prevents an easy generalisation of the notion of \emph{Parametrized BV solution}, again introduced in \cite{MielkRosSav16,MielkRosSav12}, to the inertial setting: such a notion is indeed build starting from a suitable viscous-reparametrization of the time variable. Furthermore, the rate-dependent nature of the cost forces one to consider minimization problems on an asymptotically infinite time horizon, and take the infimum over them.
	We also refer to \cite{Ago, ScilSol} for a similar analysis with no rate-independent dissipation, namely considering $\mc R=0$, where an analogous notion of solution was developed. {As it happened in \cite{ScilSol}, we prefer to consider a notion of solution which does not depend on the chosen representative of $u$ in its Lebesgue class, which is done by considering left and right limits only in the energy balance.}\par 

      {The proposed notion of solution is indeed a suitable one to extend the results of \cite{MielkRosSav16, MielkRosSav12}  to a context where the limiting effect of inertial terms is taken in consideration. We show this in our main result, Theorem \ref{mainthm}, which fulfills a twofold goal. First, we show convergence of the solutions $\xeps$ of \eqref{mainprob} to an Inertial Balanced Viscosity solution of \eqref{quasistprob}, under a quite general set of assumptions on the energy $\mc E$, which includes the one considered in \cite{GidRiv}. Secondly}, we also prove that IBV solutions can be obtained via a natural extension of the Minimizing Movements algorithm \eqref{schemevisc}, namely
		\begin{equation}\label{schemeintro}
		u^k_{\tau,\varepsilon}\in \mathop{\rm arg\,min}\limits_{x\in X}\left\{\frac{\varepsilon^2}{2\tau^2}\|x-2u_{\tau,\varepsilon}^{k-1}+u_{\tau,\varepsilon}^{k-2}\|^2_{\mathbb{M}}+ \frac{\eps}{2\tau}\|x{-}u_{\tau,\varepsilon}^{k-1}\|^2_{\mathbb{V}}+\mc R\left({x{-}u_{\tau,\varepsilon}^{k-1}}\right) + \mc E(t^{k},x)\right\},
	\end{equation}
	by sending both $\tau$ and $\eps$ to $0$. Differently from \eqref{schemevisc}, for technical reasons we need to strenghten the rate of convergence requiring $\tau/\eps^2$ to be bounded. {Furthermore, we have to require $\mc E$ to be $\Lambda$-convex (Assumption (E5) in Section \ref{sec:Setting}). Such a condition, which  amounts to require that the sum of $\mc E$ with a suitably large quadratic perturbation is convex, is quite typical in the analysis of such approximation schemes (see \cite{AGS2008}) and complies with many relevant applications. It actually allows one to have precise estimates on some rest terms in the energy balance, which arise from the iterative minimization schemes.}
	
	\subsection{Plan of the paper} In Section~\ref{sec:Setting} we fix the main notation and list the main assumptions of the paper. We also recall some basic properties of functions of bounded $\mathcal{R}$-variation (Section~\ref{subsec:Rvar}). In Section~\ref{sec:IBV} we introduce the notions of Inertial Balanced Viscosity and Inertial Virtual Viscosity solution. We also define the contact potentials (Section~\ref{sub:contpot}) and the regularized contact potentials (Section~\ref{sub:regcontpot}), while in Section~\ref{sub:costfun} we introduce the inertial cost function which will characterize the description of the jumps. Section~\ref{sec:slowload} contains the first characterization of the IBV and IVV solutions as the slow-loading limit as $\varepsilon\to0$ of dynamical solutions to \eqref{mainprob}. Finally, with Section~\ref{sec:discrete}, we derive these solutions as the limit of the time-discrete incremental variational scheme \eqref{schemeintro} as $\tau$ and $\eps$ go simultaneously to 0.
	
	\section{Notation and setting}\label{sec:Setting}
	
	Let $(X,\norm{\,\cdot\,})$ be a finite-dimensional normed vector space. We denote by $(X^*,\norm{\,\cdot\,}_*)$ the topological dual of $X$ and by $\langle w,\,v\rangle$ the duality product between $w\in X^*$ and $v\in X$. For $R>0$, we denote by $B_R$ the open ball in $X$ of radius $R$ centered at the origin, and by $\overline{B_R}$ its closure. \par
	Given any symmetric positive-semidefinite linear operator $\mathbb{Q}\colon X\to X^*$ we introduce the induced (Hilbertian) seminorm
	\begin{equation}
		|x|_\mathbb{Q}:=\scal{\mathbb{Q}x}{x}^\frac 12,
	\end{equation}
	and we denote with a capital letter $Q\ge0$ a nonnegative constant satisfying
	\begin{equation*}
		0\le |x|_\mathbb{Q}^2\le Q\|x\|^2,\quad\text{ for every }x\in X.
	\end{equation*}
	{We point out that such a constant $Q$ exists since in finite dimension any linear operator is necessarily continuous. The least $Q$ that may be chosen here is the operator norm of $\Q$, denoted by $\|\Q\|_{\rm op}$.}\par 
	If $\mathbb{Q}$ is positive-definite, the induced seminorm is actually a norm, denoted by $\|\cdot\|_\mathbb{Q}$, and, up to {possibly} enlarging the constant $Q$, there holds
	\begin{equation}
		\frac{1}{Q}\|x\|^2\le \|x\|_\mathbb{Q}^2\le Q\|x\|^2,\quad\text{ for every }x\in X.
	\end{equation}
	Furthermore the inverse operator $\mathbb{Q}^{-1}\colon X^*\to X$ induces in $X^*$ the norm 
	\begin{equation*}
		\|w\|_{\mathbb{Q}^{-1}}:=\scal{w}{\mathbb{Q}^{-1}w},
	\end{equation*} 
	which is dual to $\|\cdot\|_\mathbb{Q}$ and thus satisfies
	\begin{equation}
		|\scal{w}{v}|\le \|v\|_\mathbb{Q}\|w\|_{\mathbb{Q}^{-1}},\quad\text{ for every }w\in X^*\text{ and }v\in X.
	\end{equation}

	We briefly recall some basic definitions in convex analysis (see for instance \cite{Rocka}). Given a proper, convex, lower semicontinuous function $f\colon X\to(-\infty,+\infty]$, its (convex) subdifferential $\partial f\colon X\rightrightarrows X^*$ at a point $v\in X$ is defined as
	\begin{equation*}
		\partial f(v)=\{w\in X^*\mid f(z)\ge f(v)+\langle w,z-v\rangle,  \quad\text{for every $z\in X$}\}.
	\end{equation*}
{Notice that if $f(v)=+\infty$, then from the very definition it turns out that $\partial f(v)=\emptyset$.}\\
	The Fenchel conjugate of $f$ is the convex, lower semicontinuous function
	\begin{equation*}
		f^*\colon X^*\to (-\infty,+\infty],\quad \text{defined as }\quad f^*(w):=\sup\limits_{v\in X} \{\scal{w}{v}-f(v)\},
	\end{equation*}
	and for every $w\in X^*$ and $v\in X$ it satisfies
	\begin{equation}\label{Fenchel}
		f^*(w)+f(v)\ge \scal{w}{v},\quad\text{ with equality if and only if }\quad w\in \partial f(v).
	\end{equation}
	Given a subset $A\subset X$, we denote with $\chi_A\colon X\to [0,+\infty]$ its characteristic function , defined as
	\begin{equation*}
		\chi_A(x):=\begin{cases}
			0,&\text{if $x\in A$},\\
			+\infty, &\text{if $x\notin A$}.
		\end{cases}
	\end{equation*}
	
	\subsection{Main assumptions}
	We list below the main assumptions we will use throughout the paper.\\
	
	\noindent In the dynamic problem \eqref{mainprob} the inertial term is described by a 
	\begin{equation}\label{mass}
		\text{symmetric positive-definite linear operator $\mathbb{M}\colon X\to X^*$ ,}
	\end{equation}
	 which represents a \emph{mass distribution}.\\
	
	\noindent The possible presence of \emph{viscosity} is also considered by introducing the 
	\begin{equation}\label{viscosity}
			\text{symmetric positive-semidefinite linear operator $\mathbb{V}\colon X\to X^*$.}
	\end{equation}
	In particular, in our analysis we also include the case $\mathbb{V}\equiv 0$ (for which $|x|_\mathbb V\equiv0$), corresponding to the absence of viscous friction forces. \\
	
	Both the rate-independent \eqref{quasistprob} and the dynamic \eqref{mainprob} problems are damped by a \emph{rate-independent dissipation potential} $\mc R\colon X\to[0,+\infty)$, which models for instance dry friction. We make the following assumption:
	\begin{enumerate}[label=\textup{(R\arabic*)}, start=1]
		\item \label{hyp:R1}  the function $\mc R$ is coercive, convex, and positively homogeneous of degree one.
	\end{enumerate}
	Assumption \ref{hyp:R1} implies subadditivity, namely 
	\begin{equation*}
		\mc R (v_1+v_2)\leq \mc R(v_1)+\mc R(v_2)\,, \quad \mbox{ for every }v_1,v_2\in X\,,
	\end{equation*}
	and the existence of two positive constants $\alpha^*\ge\alpha_*>0$ for which
	\begin{equation}\label{Rbounds}
		\alpha_*\norm{v}\le \mc R(v)\le \alpha^*\norm{v},\quad\text{ for every }v\in X.
	\end{equation}
	This means that $\RR$ fails to be a norm only for the lack of symmetry.\par 
	Furthermore, since $\mc R$ is one-homogeneous, for every $v\in X$ its subdifferential $\partial\mc R(v)$ can be characterized by
	\begin{equation}
		\partial \mc R(v) = \{w\in \partial\mc R(0)\mid\,\,  \langle w,v\rangle=\mc R(v)\}\subseteq \partial\mc R(0)=:K^*.
		\label{2.2mielke}
	\end{equation}
	By \eqref{Rbounds} we notice that there holds
	\begin{equation}\label{boundedness}
		K^*\subseteq \overline{B_{\alpha^*}}.
	\end{equation}
	It is also well-known (see, e.g., \cite{Rocka}) that $K^*$ coincides with the proper domain of the Fenchel conjugate $\mc R^*$ of $\mc R$, indeed it actually holds $\mc R^*= \chi_{K^*}$.\\
	
	We finally consider the driving \emph{potential energy} $\mc E\colon [0,T]\times X\to [0,+\infty)$, which we assume to possess the following properties:
	\begin{enumerate}[label=\textup{(E\arabic*)}]
		\item \label{hyp:E1} $\mc E (\cdot,u)$ is absolutely continuous in $[0,T]$ for every $u\in X$; 
		\item \label{hyp:E2} $\mc E(t,\cdot)$ is differentiable for every $t\in[0,T]$ and the differential $D_x \mc E$ is continuous from $[0,T]\times X$ to $X^*$;
		\item \label{hyp:E3} for a.e. $t\in[0,T]$ and for every $u\in X$ it holds
		\begin{equation*}
			\left|{\partial_t} \mc E(t,u)\right|\le a(\mc E(t,u))b(t),
		\end{equation*}
		where $a\colon [0,+\infty)\to[0,+\infty)$ is nondecreasing and continuous, while $b\in L^1(0,T)$ is nonnegative;
		\item \label{hyp:E4}for every $R>0$ there exists a nonnegative function $c_R\in L^1(0,T)$ such that for a.e. $t\in[0,T]$ and for every $u_1,u_2\in {B_R}$ it holds 
		\begin{equation*}
			\left|{\partial_t}\mc E(t,u_2)-{\partial_t}\mc E(t,u_1)\right|\le c_R(t)\|u_2-u_1\|.
		\end{equation*}
	\end{enumerate}	
	We point out that the prototypical example of potential energy
	\begin{equation}\label{example}
		\mc E(t,u)=\mc U(u)-\scal{\ell(t)}{u},
	\end{equation}
	with $\mc U\in C^1(X)$ superlinear and $\ell\in W^{1,1}(0,T;X^*)$, fulfils all the previous assumptions.
	
	As noticed in \cite{GidRiv}, under these hypotheses one can prove that $\mc E$ is a continuous map, and that $t\mapsto \mc E(t,u(t))$ is absolutely continuous (resp. of bounded variation) if $u$ is absolutely continuous (resp. of bounded variation). 
	\begin{rmk}\label{rmkE}
		Thanks to \ref{hyp:E4}, it is easy to see that $D_x \mc E(\cdot,u)$ is absolutely continuous in $[0,T]$ for every $u\in X$:
		\begin{align*}
			&\qquad\|D_x \mc E(t,u)-D_x \mc E(s,u)\|_*=\scal{D_x \mc E(t,u)-D_x \mc E(s,u)}{v}\\
			&=\lim\limits_{h\to 0}\frac{\mc E(t,u{+}hv)-\mc E(t,u)-\mc E(s,u{+}hv)+\mc E(s,u)}{h}\le \liminf_{h\to 0}\frac 1h\int_{s}^{t}|\partial_t\mc E(r,u{+}hv)-\partial_t\mc E(r,u)|\d r\\
			&\le \int_{s}^{t}c_R(r)\d r\,,
		\end{align*}
		where $v\in X$ is a suitable unit vector at which the dual norm is attained, and $R$ can be chosen for instance equal to $\norm{u}+1$.
		This last property will be used in Proposition~\ref{propatomic} in order to apply a chain-rule formula for functions of bounded variation (see \cite{CrDeCicco}).
	\end{rmk}
{\begin{rmk}\label{rmk:ex}
		All the applications presented in \cite[Section~7]{GidRiv}, basically regarding masses connected with springs, are described by adding together elastic quadratic energies of the form
		\begin{equation}
			\mc E(t,u)=\frac{k}{2}(u-\ell(t))^2,
		\end{equation}
		with $k>0$ and $\ell\in W^{1,1}(0,T;\R)$. This specific form simply is the second order expansion of the real elastic potential energy of the springs
		\begin{equation}
			\mc E(t,u)=k(1-\cos(u-\ell(t))),
		\end{equation}
	which is of course nonconvex. It is however straightforward to check that it satisfies \ref{hyp:E1}-\ref{hyp:E4} (and actually also \ref{hyp:E3'}, \ref{hyp:E5} below), and thus it can be included within the framework of this paper.
\end{rmk}}
	In Section~\ref{sec:discrete}, where we deal with the discrete approximation of IBV and IVV solutions, in addition to the previous assumptions, we need to require:
	
	\begin{enumerate}[label=\textup{(E3')}]
		\item \label{hyp:E3'} the energy $\mc E$ fulfils \ref{hyp:E3} with the particular choice $a(y)=y+a_1$, for some $a_1\ge 0$;
	\end{enumerate}
	\medskip
	\begin{enumerate}[label=\textup{(E5)}]
		\item \label{hyp:E5} $\mc E(t,\cdot)$ is $\Lambda$-convex for every $t\in [0,T]$; i.e., there exists $\Lambda>0$ such that for every $t\in[0,T]$, $u_1,u_2\in X$, and every $\theta\in(0,1)$ it holds
		\begin{equation*}
			\mc E(t,(1-\theta)u_1+\theta u_2) \leq (1-\theta) \mc E(t,u_1) + \theta \mc E(t,u_2) + \frac{\Lambda}{2} \theta(1-\theta)\|u_1-u_2\|_\mathbb{I}^2\,,
		\end{equation*}
		for some symmetric positive-definite linear operator $\mathbb{I}\colon X\to X^*$.
	\end{enumerate}
	We notice that by \ref{hyp:E3'} and Gronwall's lemma we can infer
	\begin{equation*}
		\mc E(t,u)+a_1\leq (\mc E(s,u)+a_1){e}^{\int_{s}^{t}b(r)\d r}\,, \quad\text{for every }0\le s\le t\le T,
	\end{equation*}
	whence
	\begin{equation}\label{rmkLip}
		|\partial_t\mc E(t,u)|\leq (\mc E(s,u)+a_1)b(t){e}^{\int_{s}^{t}b(r)\d r}\,, \quad\text{for every }0\le s\le t\le T.
	\end{equation}
	It is also easy to check that \ref{hyp:E5} implies
	\begin{equation}
		\langle D_x\mc E(t,u_1),u_2-u_1\rangle\leq \mc E(t,u_2)-\mc E(t,u_1)+\frac{\Lambda}{2}\|u_1-u_2\|_\mathbb{I}^2\,,\,\, \mbox{ for every }t\in[0,T],\, u_1,u_2\in X\,.
		\label{eq:assumption}
	\end{equation}
	Indeed, by using the mean value theorem, for some $\zeta\in [0,1]$ we have
	\begin{equation*}
		\begin{split}
			&\quad\theta \left[\mc E(t,u_2)-\mc E(t,u_1)+\frac{\Lambda}{2} (1-\theta)\|u_1-u_2\|_\mathbb{I}^2\right]\\  &\geq \mc E(t,(1-\theta)u_1+\theta u_2)-\mc E(t,u_1)
			= \theta \langle D_x\mc E(t,u_1+\zeta\theta(u_1-u_2)),u_2-u_1\rangle,
		\end{split}
	\end{equation*}
	whence \eqref{eq:assumption} follows up to simplifying $\theta$ in both sides and then letting $\theta\to0$.
	
	We finally point out that an energy $\mc E$ as in \eqref{example} always complies with \ref{hyp:E3'}, while it fulfils \ref{hyp:E5} if in addition $\mc U$ is $\Lambda$-convex.
	
	\subsection{Functions of bounded $\mathcal{R}$-variation} \label{subsec:Rvar}
	
	We recall here a suitable generalization of functions of bounded variation useful to deal with functions satisfying \ref{hyp:R1}.
	
	\begin{defi}\label{rvar}
		Given a function $f:[a,b]\to X$, we define the \emph{pointwise $\mathcal{R}$-variation} of $f$ in $[s,t]$, with $a\leq s<t\leq b$, as
		\begin{equation*}
			V_{\mathcal{R}}(f;s,t):=\sup\left\{\sum_{k=1}^n\mathcal{R}(f(t_k)-f(t_{k-1}))\mid\, s=t_0<t_1<\dots<t_{n-1}<t_n=t\right\}\,.
		\end{equation*}
		We also set $V_{\mathcal{R}}(f;t,t):=0$ for every $t\in[a,b]$.\par 
		We say that $f$ is a \emph{function of bounded $\mathcal{R}$-variation} in $[a,b]$, and we write $f\in BV_{\mathcal{R}}([a,b];X)$, if its $\mathcal{R}$-variation in $[a,b]$ is finite; i.e., $V_{\mathcal{R}}(f;a,b)<+\infty$.
	\end{defi}
	
	Notice that, by virtue of \eqref{Rbounds}, we have $f\in BV_{\mathcal{R}}([a,b];X)$ if and only if $f\in BV([a,b];X)$ in the classical sense. In particular, $f\in BV_{\mathcal{R}}([a,b];X)$ is regulated, i.e., it admits left and right limits at every $t\in[a,b]$:
	\begin{equation*}
		f^+(t):=\lim_{t_j\searrow t}f(t_j), \quad \mbox{ and }\quad f^-(t):=\lim_{t_j\nearrow t}f(t_j)\,,
	\end{equation*}
	with the convention $f^-(a):=f(a)$ and $f^+(b):=f(b)$. Moreover, its pointwise jump set $J_f$ is at most countable.
	
	{It is well known (see, e.g., \cite{AFP}) that $f$ can be uniquely decomposed as follows:
      \begin{equation}
     f=f_{\mathcal{L}}+f_{\rm Ca}+ f_{J}
      \label{eq:decompositionf}
	\end{equation}
   with $f_{\mathcal{L}}$ being an absolutely continuous function, $f_{\rm Ca}$ a continuous Cantor-type function, and $f_J$ a jump function.
If we denote by $f'$ the distributional derivative of $f\in BV_\RR([a,b];X)$, and recall that $f'$ is a Radon vector measure with finite total variation $|f'|$,} it follows that $f'$ can be decomposed into the sum of the three mutually singular measures
	\begin{equation}
		f'=f'_{\mathcal{L}}+f'_{\rm Ca}+f'_{\rm J}\,,\quad f'_{\mathcal{L}}=\dot{f}\mathcal{L}^1\,,\quad f'_{\rm co}:=f'_{\mathcal{L}}+f'_{\rm Ca}\,.
		\label{eq:decompositionm}
	\end{equation}
	In \eqref{eq:decompositionm}, $f'_{\mathcal{L}}$ is the absolutely continuous part with respect to the Lebesgue measure $\mathcal{L}^1$, whose Lebesgue density $\dot{f}$ is the usual pointwise ($\mathcal{L}^1$-a.e. defined) derivative, $f'_{\rm J}$ is the jump part concentrated on the essential jump set of $f$
	\begin{equation*}
		J^{\rm e}_f:=\{t\in [a,b]\mid f^+(t)\neq f^-(t)\}\subseteq J_f,
	\end{equation*}
	and $f'_{\rm Ca}$ is the Cantor part, such that $f'_{\rm Ca}(\{t\})=0$ for every $t\in[0,T]$. The measure $f'_{\rm co}$ is the diffuse part of the measure, and does not charge $J_f$. {The functions $f_{\mathcal{L}}$, $f_{\rm Ca}$, and $f_{J}$ in \eqref{eq:decompositionf} are exactly the distributional primitives of the measures $f'_{\mathcal{L}}$, $f'_{\rm Ca}$, and $f'_{J}$ in \eqref{eq:decompositionm}. We will use the notation $f_{\rm co}$ to denote the continuous part of $f$, that is
$f_{\rm co}=f_{\mathcal{L}}+f_{\rm Ca}$}.
	
	We also remark that, for $a\le s\le t\le b$, the function $V_{\mathcal{R}}(f;s,t)$ is monotone in both entries, hence it makes sense to consider the limits $V_{\mathcal{R}}(f;s-,t+)$. The following formula (see for instance \cite[Section 2]{MielkRosSav12}) relates $V_{\mathcal{R}}(f;s-,t+)$ with the distributional derivative of $f$, up to the jump part which is depending on the pointwise behavior of $f$. Setting $\lambda=\mathcal L^1+ |f'_{\rm Ca}|$, it namely holds
	\begin{equation}\label{eq: representation}
		V_{\mathcal{R}}(f;s-,t+)=\int_{s}^{t}	\RR\left(\frac{\d f'_{\rm co}}{\d\lambda}(r)\right)\d\lambda(r)+\!\!\!\!\sum_{r\in J_f\cap[s,t]}\left(\RR(f^+(r){-}f(r))+\RR(f(r){-}f^-(r))\right),
	\end{equation}
	where $\frac{\d f'_{\rm co}}{\d\lambda}$ is the Radon-Nikodym derivative. Observe that, by the positive one-homogeneity of $\RR$, actually any measure $\nu$ such that $f'_{\rm co}<<\nu$ can replace $\lambda$ in the integral term at the right-hand side.\\
	{It follows from \eqref{eq: representation} that  the continuous part of the $\mc R$-variation of $f$ agrees with the $\mc R$-variation of $f_{\rm co}$ and satisfies}
	\begin{equation}\label{contvar}
		V_\mc R(f_{\rm co};s,t)=\int_{s}^{t}	\RR\left(\frac{\d f'_{\rm co}}{\d\lambda}(r)\right)\d\lambda(r).
	\end{equation}
	We finally notice that by dropping the pointwise value of $f$ at jump points (by the subadditivity of $\RR$), and only considering the essential jumps, we are led to the so-called essential $\RR$-variation
	\begin{equation}\label{essvar}
		\RR(f')([s,t]):=V_\mc R(f_{\rm co};s,t)+\sum_{r\in J^{\rm e}_f\cap[s,t]}\RR(f^+(r)-f^-(r))\le V_{\mathcal{R}}(f;s-,t+).
	\end{equation}
	The term $\RR(f')$ actually defines a Radon measure {(see \cite{GofSer})}, which generalizes the concept of total variation $|f'|$ (corresponding to the particular choice $\RR(\cdot)=\norm{\,\cdot\,}$).

	\section{Inertial Balanced Viscosity and Inertial Virtual Viscosity solutions}\label{sec:IBV}
	In this section we rigorously introduce the notions of Inertial Balanced Viscosity and Inertial Virtual Viscosity solution. We also state our main result, see Theorem~\ref{mainthm}, postponing its proof to the forthcoming sections. \par 
	As in the vanishing-viscosity approach of \cite{MielkRosSav12}, the starting point consists in an alternative formulation of the dynamic problem \eqref{mainprob} based on the so--called De Giorgi's energy-dissipation principle (see the pioneering work \cite{DeG80} and other applications in \cite{Mielk, RosThom}). Roughly speaking, the idea is to keep together all the dissipative terms appearing in the dynamic model, namely viscosity and rate-independent dissipation; we are thus led to consider the functional
	\begin{equation}
		{\mc R}_\varepsilon(v):=\mc R (v) + \frac{\varepsilon}{2}|v|_{\mathbb{V}}^2\,.
		\label{scaleddissip}
	\end{equation}
	It is then easy to check (see, e.g., \cite[p. 47]{MielkRosSav12}) that the subdifferential of $\mathcal{R}_\varepsilon$ is explicitly given by
	\begin{equation*}
		\partial {\mc R}_\varepsilon(v) = \partial \mc R (v) + \varepsilon \mathbb{V}v\,,
	\end{equation*}
	so the dynamic problem \eqref{mainprob} can be rewritten as
	\begin{equation}
		\partial {\mc R}_\varepsilon (\dot{u}^\varepsilon(t))\ni - \eps^2 \mathbb{M}\xepsdd(t) - D_x\mc E(t,\xeps(t))=: w^\eps(t),\quad \mbox{for a.e. $t\in[0,T]$}\,.
		\label{mainprobequiv}
	\end{equation}
	By using \eqref{Fenchel}, and exploiting the classical chain-rule formula for $\mc E$, one obtains that the dynamic problem \eqref{mainprobequiv} is actually equivalent to the augmented energy balance
	\begin{equation}\label{augmeb}
		\begin{split}
			&\quad\frac{\varepsilon^2}{2}\|\xepsd(t)\|_{\mathbb{M}}^2+ \mc E(t,\xeps(t))+\int_s^t {\mc R}_\eps(\xepsd(r))+\RR_\eps^*(w^\eps(r))\,\mathrm{d}r \\
			&=\frac{\varepsilon^2}{2}\|\xepsd(s)\|_{\mathbb{M}}^2 +  \mc E(s,\xeps(s)) +  \int_s^t \partial_t\mc E(r,\xeps(r))\,\mathrm{d}r,\quad \text{for every } 0\le s\le t\le T.
		\end{split}
	\end{equation}

		We point out that in our case of additive viscosity \eqref{scaleddissip}, the Fenchel conjugate $\mathcal{R}^*_\varepsilon$ can be explicitly computed by means of the inf-sup convolution formula (see, e.g., \cite[§12]{Rocka}) and turns out to be 
		\begin{equation}
			\mathcal{R}^*_\varepsilon(w)=\begin{cases}\displaystyle
				\frac{1}{2\eps} \inf\limits_{\substack{z\in K^* \\ w-z\in (\ker\V)^\perp}}\scal{w-z}{\V'(w-z)},&\text{if }w\in K^*+(\ker\V)^\perp,\\
				+\infty,&\text{otherwise,}
			\end{cases}
			\label{eq:conjugates}
		\end{equation}
		where
		\begin{equation*}
			(\ker\V)^\perp=\{w\in X^*\mid \scal{w}{v}=0 \text{ for every }v\in \ker\V\},
		\end{equation*}
		{is the \emph{annihilator} of $\ker \V$} and $\V'\colon (\ker\V)^\perp\to X$ is the inverse of the operator $\V$ restricted to (the identification of) $(\ker\V)^\perp$ (in $X$).\par 
		In particular, in the two extreme situations $\V=0$ and $\V$ positive-definite we get, respectively:
		\begin{equation*}
			\RR_\eps^*(w)=\RR^*(w)=\chi_{K^*}(w),\quad\quad\RR_\eps^*(w)=	\frac{1}{2\eps}\dist^2_{\V^{-1}}(w,K^*),
		\end{equation*}
		where
		\begin{equation*}
			\dist_{\V^{-1}}(w,K^*):=\inf\limits_{z\in K^*}\|w-z\|_{\V^{-1}},
		\end{equation*}
	 	denotes the distance from $K^*$, measured with respect to the norm $\|\cdot\|_{\V^{-1}}$.

	\subsection{Contact potentials} \label{sub:contpot}
	
	The energy balance \eqref{augmeb} naturally leads to the introduction of a so-called (viscous) contact potential associated to $\V$. In the spirit of \cite{MielkRosSav12} and taking into account \eqref{eq:conjugates}, we thus define:
	\begin{defi}\label{contpot}
		The (viscous) \emph{contact potential} related to the viscosity operator $\V$ is the map $p_\V\colon X\times X^*\to [0,+\infty]$ defined as
		\begin{equation*}
			p_\V(v,w):=\inf_{\eps>0}\left(\RR_\eps(v)+\RR_\eps^*(w)\right)=\begin{cases}
				\RR(v)+|v|_\V\inf\limits_{\substack{z\in K^* \\ w-z\in (\ker\V)^\perp}}|w-z|_{\V'},&\text{if }w\in K^*+(\ker\V)^\perp,\\
				+\infty,&\text{otherwise}.
			\end{cases}
		\end{equation*}
	\end{defi}
	In the two extreme situations $\V=0$ and $\V$ positive-definite we get, respectively:
	\begin{equation}\label{explicit}
		p_0(v,w)=\RR(v)+\chi_{K^*}(w),\quad\quad p_\V(v,w)=\RR(v)+\norm{v}_\V\dist_{\V^{-1}}(w,K^*).
	\end{equation}
	Therefore, in the positive-definite case we retrieve the vanishing-viscosity contact potential defined in \cite{MielkRosSav12}. \\	
	By the explicit formula we easily infer the following properties for the contact potential $p_\V$:
	
	\begin{enumerate}
		\item $p_\V(\cdot,w)$ is positively one-homogeneous and convex, for every $w\in X^*$;
		\item $p_\V(v,\cdot)$ is convex, for every $v\in X$;
		\item $p_\V(v,w)\geq\max\{\RR(v),\langle w,v\rangle\}$, for every $v\in X$ and $w\in X^*$;
		\item $p_\V(0,w)=\chi_{K^*+(\ker\V)^\perp}(w)$, and $p_\V(v,0)=\RR(v)$.
	\end{enumerate}	
	Furthermore we also observe that:
	\begin{itemize}
		\item[(5)] $p_\V(\cdot,w)$ is symmetric for every $w\in X^*$ if and only if $\RR$ is symmetric. 
	\end{itemize}\medskip
	
	At this stage a warning is mandatory: we point out that our potential $p_\V$ in general can take the value $+\infty$, due to the semi-definiteness of the viscosity operator $\V$. This feature does not appear in \cite{MielkRosSav12}, where indeed a full viscosity is always present and the contact potential is continuous and finite. This difference will create serious issues in the forthcoming analysis, leading to the original notion of Inertial Virtual Viscosity; we are thus led to couple $p_\V$ with a ``regularized'' contact potential $p$, as follows.
	\begin{defi}\label{RCP}
		We say that a \emph{continuous} map $p\colon X\times X^*\to [0,+\infty)$ is a \emph{regularized contact potential} with respect to $p_\V$, and we write $p\in {RCP}_\V$, if:
		\begin{itemize}
			\item[(i)] $p(\cdot,w)$ is positively one-homogeneous, for every $w\in X^*$;
			\item[(ii)] $p(v,\cdot)$ is convex,  for every $v\in X$;
			\item[(iii)] $\max\{\RR(v),\langle w,v\rangle\}\leq p(v,w)\leq p_\V(v,w)$, for every $v\in X$ and $w\in X^*$;
			\item[(iv)] there exists a positive constant $L>0$ such that\begin{equation*}
				|p(v,w_1)-p(v,w_2)|\le L\norm{v}\norm{w_1-w_2}_*,\quad \text{ for every }v\in X, \text{ and }w_1,w_2\in X^*.
			\end{equation*}
		\end{itemize}
	\end{defi}
	\begin{rmk}\label{regpV}
		In the case $\V$ positive-definite, the contact potential $p_\V$ itself belongs to $RCP_\V$. This easily follows by the explicit form \eqref{explicit}. Observe that in this case $p_\V$ takes only finite values.
	\end{rmk}
	
	Notice that in the above definition we are not requiring the convexity of $p$ with respect to the variable $v$. The convexity in the second variable $(ii)$, instead, will be crucial in Proposition~\ref{propmin}.\par 
	We also point out that the main property of regularized contact potentials, missing in general for $p_\V$, is the weighted Lipschitzianity $(iv)$ with respect to the second variable: this will be heavily used in Proposition~\ref{proposizione1}.\\
	We finally observe that by $(iii)$ and $(iv)$ any $p\in RCP_\V$ satisfies
	\begin{align}\label{cont0}
		p(v,w)&\le |p(v,w)-p(v,0)|+p(v,0)\le L\norm{v}\norm{w}_*+p_\V(v,0)=L\norm{v}\norm{w}_*+\RR(v)\nonumber\\
		&\le (\alpha^*+L\norm{w}_*)\norm{v},
	\end{align}
	where we exploited \eqref{Rbounds}. In particular it holds
	\begin{equation*}
		p(0,w)=0,\quad\quad\text{for every }w\in X^*.
	\end{equation*}
	
	\subsection{Parametrized families of regularized contact potentials}\label{sub:regcontpot}
	With the following result, we show that a whole family of regularized contact potentials can be constructed by means of a suitable version of the Yosida transform.
	
	For every $\lambda\ge 1$ and every symmetric positive-definite linear operator $\U\colon X\to X^*$, we define the function $p^{\lambda,\U}_\V\colon X\times X^*\to [0,+\infty)$ as
	\begin{equation}
		p^{\lambda,\U}_\V(v,w):=\inf_{\eta\in X^*}\left\{p_\V(v,\eta)+\lambda\norm{v}_\U\norm{w-\eta}_{\U^{-1}}\right\}\,,\quad v\in X\,,\,\, w\in X^*\,.
		\label{eq:regularizedpot}	
	\end{equation}
	
	\begin{prop}\label{propYosida}
		Let $\lambda\ge 1$ and $\U$ be a symmetric positive-definite linear operator. Then $p^{\lambda,\U}_\V\in RCP_\V$.
		Furthermore, for every $v\neq 0$, one has
		\begin{equation}\label{supYosida}
			p_\V(v,w)=\sup_{\lambda\ge 1} 	p^{\lambda,\U}_\V(v,w)=\lim_{\lambda\to +\infty}	p^{\lambda,\U}_\V(v,w).
		\end{equation}
		If in addition $\RR$ is symmetric, then for every $w\in X^*$ the function $p^{\lambda,\U}_\V(\cdot,w)$ is symmetric as well.
	\end{prop}
	\begin{proof}
		We first notice that $p^{\lambda,\U}_\V$ has nonnegative finite values since $p_\V$ is not identically $+\infty$ and it is nonnegative. Moreover \eqref{supYosida} is a standard property of Yosida transform (notice that $\lambda\norm{v}_\U>0$ if $v\neq 0$). Also, if $\RR$ is symmetric, symmetry of $p^{\lambda,\U}_\V(\cdot,w)$ is a straightforward byproduct of \eqref{eq:regularizedpot} since in this case $p_\V(\cdot,w)$ is symmetric.\par
		Now, we have to confirm the properties $(i)-(iv)$ of Definition~\ref{RCP}. Property $(i)$ follows by the one-homogeneity of $p_\V(\cdot,\eta)$ and of the norm.\par 
		Property $(ii)$ follows since the Yosida transform of a convex function is convex, and $p_\V (v,\cdot)$ is convex.\par 
		The right inequality in $(iii)$ is obtained by choosing $\eta=w$ in the definition of $p^{\lambda,\U}_\V$, while the left-one follows from the fact that $p_\V(v,\cdot)\ge \RR(v)$ combined with the simple inequality
		\begin{equation*}
			\scal{w}{v}=\scal{\eta}{v}+\scal{w-\eta}{v}\le p_\V(v,\eta)+\lambda\norm{v}_\U\norm{w-\eta}_{\U^{-1}}.
		\end{equation*}
		Property $(iv)$ is again a straightforward consequence of the Yosida transform: one can choose $L=\lambda\sqrt{\norm{\U}_{\rm op}\norm{\U^{-1}}_{\rm op}}$.\par
		We are only left to prove that $p^{\lambda,\U}_\V$ is continuous. We first observe that thanks to $(iv)$ it is enough to prove that $p^{\lambda,\U}_\V(\cdot,w)$ is continuous for every fixed $w\in X^*$. The continuity in $v=0$ follows easily by \eqref{cont0}; if $v\neq 0$ we need more work. We make the following claims:\\
		\textbf{CLAIM 1)} There exists a positive constant $C_1>0$ such that 
		\begin{equation}
			p_\V(v_1,w)\le p_\V(v_2,w)+C_1(1+\norm{w}_*)\norm{v_1-v_2},\quad \text{for every }v_1,\,v_2\in X,\text{ and }w\in X^*.
			\label{eq:stimapv}
		\end{equation}
		\textbf{CLAIM 2)} If $v\neq 0$, then there exists a positive constant $C_2>0$ such that 
		\begin{equation*}
			p^{\lambda,\U}_\V(v,w)=\inf\limits_{\substack{\eta\in {X^*}\\\norm{\eta}_*\le C_2 \rho(\norm{v},\norm{w}_*)}}\left\{p_\V(v,\eta)+\lambda\norm{v}_\U\norm{w-\eta}_{\U^{-1}}\right\},
		\end{equation*}
		where $\rho(\norm{v},\norm{w}_*):=1+\norm{w}_*+1/\norm{v}$.\\
		\textbf{CLAIM 3)} There exists a positive constant $C_3>0$ such that 
		\begin{equation*}
			|p^{\lambda,\U}_\V(v_1,w)-p^{\lambda,\U}_\V(v_2,w)|\le C_3 \max\{\rho(\norm{v_1},\norm{w}_*),\rho(\norm{v_2},\norm{w}_*)\}\norm{v_1-v_2},
		\end{equation*}
		for every $v_1,v_2\in X\setminus\{0\}$, and $w\in X^*$.\par
		
		From CLAIM 3) we easily deduce the continuity of $p^{\lambda,\U}_\V(\cdot,w)$ in $v\neq 0$, thus we only need to prove its validity.\par
		We start with CLAIM 1), and we observe that it is enough to prove it for $w\in K^*+(\ker\V)^\perp$. By exploiting the subaddivity of $\RR$ together with \eqref{Rbounds} and \eqref{boundedness} we easily obtain
		\begin{align*}
			p_\V(v_1,w)&=\RR(v_1)+|v_1|_\V\inf\limits_{\substack{z\in K^* \\ w-z\in (\ker\V)^\perp}}|w-z|_{\V'}\\
			&\le \RR(v_2)+\alpha^*\norm{v_1-v_2}+\big(|v_2|_\V+\sqrt{V}\norm{v_1-v_2}\big)\inf\limits_{\substack{z\in K^* \\ w-z\in (\ker\V)^\perp}}|w-z|_{\V'}\\
			&\le \RR(v_2)+|v_2|_\V\inf\limits_{\substack{z\in K^* \\ w-z\in (\ker\V)^\perp}}|w-z|_{\V'}+\big(\alpha^*+\sqrt{VV'}\norm{w}_*+\sqrt{VV'}\alpha^*\big)\norm{v_1-v_2}\\
			&=p_\V(v_2,w)+\big(\alpha^*+\sqrt{VV'}\norm{w}_*+\sqrt{VV'}\alpha^*\big)\norm{v_1-v_2},
		\end{align*}
		and CLAIM 1) is proved.\par 
		To prove CLAIM 2) it is enough to show that an infimizing sequence $\{\eta_j\}_{j\in\N}$ for $p^{\lambda,\U}_\V(v,w)$ is uniformly bounded by $\rho(\norm{v},\norm{w}_*)$, up to a multiplicative constant. Being an infimizing sequence, $\eta_j$ satisfies
		\begin{equation*}
			1+p^{\lambda,\U}_\V(v,w)\ge p_\V(v,\eta_j)+\lambda\norm{v}_\U\norm{w-\eta_j}_{\U^{-1}}\ge \norm{v}_\U\norm{w-\eta_j}_{\U^{-1}}, 
		\end{equation*}
		and thus, by using \eqref{cont0}, we infer
		\begin{align*}
			\norm{\eta_j}_*&\le \norm{w}_*+C\norm{w-\eta_j}_{\U^{-1}}\le \norm{w}_*+C \frac{1+p^{\lambda,\U}_\V(v,w)}{\norm{v}}\\
			&\le \norm{w}_*+C \frac{1+ (\alpha^*+L\norm{w}_*)\norm{v}}{\norm{v}}\le C \rho(\norm{v},\norm{w}_*).
		\end{align*}
		We now need to prove CLAIM 3). To this aim we take $\eta\in X^*$ such that $\norm{\eta}_*\le C_2 \rho(\norm{v_2},\norm{w}_*)$ and we estimate exploiting CLAIM 1):
		\begin{align*}
			p^{\lambda,\U}_\V(v_1,w)&\le p_\V(v_1,\eta)+\lambda\norm{v_1}_\U\norm{w-\eta}_{\U^{-1}}\\
			&\le p_\V(v_2,\eta)+C_1(1+\norm{w}_*)\norm{v_1{-}v_2}+\lambda \norm{v_2}_\U\norm{w{-}\eta}_{\U^{-1}}+\lambda \norm{w{-}\eta}_{\U^{-1}}\norm{v_1{-}v_2}_{\U}\\
			&\le p_\V(v_2,\eta)+\lambda \norm{v_2}_\U\norm{w{-}\eta}_{\U^{-1}}+ C\big(1+\norm{w}_*+\norm{\eta}_*\big)\norm{v_1-v_2}\\
			&\le p_\V(v_2,\eta)+\lambda \norm{v_2}_\U\norm{w{-}\eta}_{\U^{-1}}+C  \rho(\norm{v_2},\norm{w}_*)\norm{v_1-v_2}
		\end{align*}
		By using CLAIM 2), from the above inequality we deduce
		\begin{equation*}
			p^{\lambda,\U}_\V(v_1,w)\le p^{\lambda,\U}_\V(v_2,w)+C  \rho(\norm{v_2},\norm{w}_*)\norm{v_1-v_2}.
		\end{equation*}
		By interchanging the role of $v_1$ and $v_2$ we thus complete the proof of CLAIM 3) and we conclude.
	\end{proof}
	We notice that in the case $\V$ positive-definite, the contact potential $p_\V$ coincides with its Yosida approximation, if we choose $\U=\V$:
	\begin{equation}
		p^{\lambda,\V}_\V(v,w)=p_\V(v,w),\quad \text{for every }\lambda\ge 1,\,v\in X,\text{ and }w\in X^*.
	\end{equation}
	Indeed by means of the explicit formula \eqref{explicit}, it holds
	\begin{align*}
		p_\V(v,w)&\ge p^{\lambda,\V}_\V(v,w)\ge p^{1,\V}_\V(v,w)=\inf_{\eta\in X^*}\left\{p_\V(v,\eta)+\norm{v}_\V\norm{w-\eta}_{\V^{-1}}\right\}\\
		&= \inf_{\eta\in X^*}\left\{\RR(v)+\norm{v}_\V\left(\dist_{\V^{-1}}(\eta,K^*)+\norm{w-\eta}_{\V^{-1}}\right)\right\}\\
		&=\RR(v)+\norm{v}_\V\inf_{\eta\in X^*}\left\{\dist_{\V^{-1}}(\eta,K^*)+\norm{w-\eta}_{\V^{-1}}\right\}\\
		&\ge \RR(v)+\norm{v}_\V\dist_{\V^{-1}}(w,K^*)=p_\V(v,w),
	\end{align*}
	where the last inequality is a simple byproduct of the triangle inequality. This fact corroborates Remark~\ref{regpV}.\par 
	In the opposite situation $\V=0$ it is not difficult to see that the Yosida transform takes a more explicit form:
	\begin{equation}\label{virtualmotivation}
		p^{\lambda,\U}_0(v,w)=\mc R(v)+\lambda\norm{v}_\U\dist_{\U^{-1}}(w,K^*).
	\end{equation}
	Compare this last formula with \eqref{explicit}, the case of $\V$ positive-definite.

	\subsection{The inertial energy-dissipation cost}\label{sub:costfun}
	Once the notion of contact potential has been developed, we are in a position to rigorously introduce the cost function which will govern the jump transient of IBV and IVV solutions. The crucial difference with respect to the vanishing-viscosity cost of BV solutions \cite{MielkRosSav12} is its rate-dependent nature, caused by the term $\M \ddot{v}$ inside the integral which is reminiscent of the original inertial effects.
	
	\begin{defi}\label{defcost}
		For every $t\in[0,T]$ and $u_1,u_2\in X$, we define the \emph{inertial energy-dissipation cost} related to $p\in RCP_\V\cup\{p_\V\}$ as
		\begin{equation}
			c^{\mathbb{M},p}_t(u_1,u_2):=\inf\left\{\int_{-N}^{N} p(\dot{v}(r),-\mathbb{M}\ddot{v}(r)-D_x\mc E(t,v(r)))
			\,\mathrm{d}r\mid \,\, N\in\N,\, v\in V^{\mathbb{M},N}_{u_1,u_2} \right\}\,,
			\label{eq:cost}
		\end{equation}
		where
		\begin{equation*}
			V^{\mathbb{M},N}_{u_1,u_2}\!:=\!\left\{\!v\in W^{2,\infty}({-}N,N;X)\mid v({-}N)\!=\!u_1, v(N)\!=\!u_2,\, \M\dot{v}(\pm N)\!=\!0,\, \essup_{r\in [-N,N]}\|\M\ddot{v}(r)\|_* \!\le\! \overline{C}\!\right\},
		\end{equation*}
		denotes the class of the \emph{admissible curves} and $\overline{C}$ is the constant of Proposition~\ref{propbounds} and Corollary~\ref{prop:inequality1} (depending only on the data of the problem).\par 
		We also define the inertial cost directly related to the viscosity operator $\V$ by taking the supremum among the costs over all the regularized contact potentials:
		\begin{equation}
			c^{\mathbb{M},\mathbb{V}}_t(u_1,u_2):=\sup_{p\in RCP_\V} c^{\mathbb{M},p}_t(u_1,u_2).
		\end{equation}
	\end{defi}
	
	\begin{rmk}
		We point out that in the case $\V$ positive-definite Remark~\ref{regpV} yields
		\begin{equation}\label{costpv}
				c^{\mathbb{M},\mathbb{V}}_t(u_1,u_2)=	c^{\mathbb{M},p_\mathbb{V}}_t(u_1,u_2).
		\end{equation}
	This is consistent with the vanishing-viscosity analysis performed in \cite{MielkRosSav12}, in which the cost is (formally) equivalent to \eqref{costpv} by taking $\M\equiv 0$ (see \eqref{costvisc}).
	\end{rmk}
		A relevant feature of the inertial cost is that the value $c^{\mathbb{M},p}_t(u_1,u_2)$ provides an upper bound for the energy gap $\mc E(t,u_1) - \mc E(t,u_2)$ for every $p\in RCP_\V$, as shown with the following proposition.
	
	\begin{prop}\label{thmupperbound}
		For every $t\in[0,T]$ and $u_1,u_2\in X$ we have
		\begin{equation*}
			\mc E(t,u_1) - \mc E(t,u_2) \leq \inf_{p\in RCP_\V}c^{\mathbb{M},p}_t(u_1,u_2)\,.
		\end{equation*}
	\end{prop}
	
	\proof
	Fix $p\in RCP_\V$ and let $N\in\N$ and $v\in V^{\mathbb{M},N}_{u_1,u_2}$. From the fundamental theorem of calculus and property $(iii)$ of regularized contact potentials we deduce:
	\begin{equation*}
		\begin{split}
			\mc E(t,u_1)-\mc E(t,u_2) & = \mc E(t,v(-N))-\mc E(t,v(N)) + \frac{1}{2} \|\dot{v}(-N)\|_{\mathbb{M}}^2 - \frac{1}{2} \|\dot{v}(N)\|_{\mathbb{M}}^2 \\
			& = \int_{-N}^{N} \langle -\mathbb{M}\ddot{v}(r)-D_x\mc{E}(t,v(r)), \dot{v}(r)\rangle\,\mathrm{d}r \\
			& \leq \int_{-N}^{N} p(\dot{v}(r),-\mathbb{M}\ddot{v}(r)-D_x\mc E(t,v(r)))
			\,\mathrm{d}r,
		\end{split}
	\end{equation*}
	and the assertion follows by the arbitrariness of $v$, $N$ and $p$.
	\endproof
	
	With the notion of inertial energy-dissipation cost at hand we can give the definition of Inertial Virtual Viscosity and Inertial Balanced Viscosity solutions.
	
	\begin{defi}
		We say that a function $u\in BV_\mc R([0,T];X)$ is an \emph{Inertial Virtual Viscosity} (IVV) solution to the rate-independent system \eqref{quasistprob}, related to $\M$ and $\V$, if it complies both with the local stability condition
		\begin{equation}\label{locstab}
			-D_x\mc E(t,u(t))\in K^*,\quad\text{for every }t\in [0,T]\setminus J_u,
		\end{equation}
		and the energy balance 
		\begin{equation}\label{eb}
			\mc E(t,u^+(t))+V_\RR(u_{\rm co};s,t)+\sum_{r\in J^{\rm e}_u\cap[s,t]}\!\!\! c^{\M,\V}_r(u^-(r),u^+(r))=	\mc E(s,u^-(s))+\int_{s}^{t}\partial_t\mc E(r,u(r))\d r\,,
		\end{equation}
		for every $0\le s\le t\le T$.\par 
		If $\V$ is positive-definite, in which case \eqref{eb} is satisfied with $c^{\M,p_\mathbb{V}}$ in place of $c^{\M,\mathbb{V}}$ (see \eqref{costpv}), we say that $u$ is an \emph{Inertial Balanced Viscosity} (IBV) solution.
	\end{defi}
	\begin{rmk}
		By Proposition~\ref{thmupperbound} we deduce that for any IVV solution there holds
		\begin{equation*}
			c^{\M,\V}_t(u^-(t),u^+(t))=\sup_{p\in RCP_\V}c^{\mathbb{M},p}_t(u^-(t),u^+(t))=\inf_{p\in RCP_\V}c^{\mathbb{M},p}_t(u^-(t),u^+(t)),\text{ for every }t\in [0,T],
		\end{equation*}
		and thus in \eqref{eb} we can actually replace $c^{\M,\V}$ with $c^{\mathbb{M},p}$ for an arbitrary $p\in RCP_\V$.\par 
		It is however not clear whether we can replace it with the cost related to the contact potential $p_\V$ itself (i.e. $c^{\mathbb{M},p_\V}$) in the case of a generic $\V$ semidefinite, despite Proposition~\ref{propYosida} shows that $p_\V$ can be always approximated (except for $v=0$) by suitable regularized contact potentials.
	\end{rmk}
	The term ``virtual'' in the definition of IVV solutions is motivated by the presence of the Yosida-type potentials \eqref{eq:regularizedpot} inside the set of regularized contact potentials $RCP_\V$. They are indeed constructed by means of a symmetric positive-definite linear operator $\mathbb{U}$, which plays the role of a virtual viscosity, since a priori it is not present in the problem under study (see in particular \eqref{virtualmotivation}). The term ``balanced'' for IBV solutions is instead inherited from \cite{MielkRosSav12}.
	
	The main result of the paper can now be stated as follows.
	\begin{thm}\label{mainthm}
		Let $\M,\V$ satisfy \eqref{mass}, \eqref{viscosity} and assume \ref{hyp:E1}--\ref{hyp:E4}, and \ref{hyp:R1}. Let $u_0^\varepsilon\to u_0$, $\eps u_1^\eps\to 0$. Then the following two assertions hold true:
		\begin{itemize}
			\item[\rm (I)] for every sequence $\epsj\to 0$ there exists a subsequence (not relabelled) along which the sequence of dynamic solutions $\xepsj$ to \eqref{mainprob} pointwise converges to an Inertial Virtual Viscosity solution of the rate-independent system \eqref{quasistprob};
			\item[\rm (II)] assume in addition \ref{hyp:E3'} and \ref{hyp:E5}; then for every sequence $(\tau_j,\epsj)\to (0,0)$ satisfying
			\begin{equation*}
				\sup\limits_{j\in\N}\frac{\tau_j}{\varepsilon_j^2}<+\infty,
			\end{equation*} 
			there exists a subsequence (not relabelled) along which the sequence of piecewise affine interpolants $\widehat{u}_{\tau_j,\epsj}$, defined in \eqref{affineinterpolant} and coming from the Minimizing Movements scheme \eqref{schemeincond}, pointwise converges to an Inertial Virtual Viscosity solution of the rate-independent system \eqref{quasistprob}.
		\end{itemize} 
		In both cases, the limit function is an Inertial Balanced Viscosity solution if $\V$ is positive-definite.
	\end{thm}
	
	The proof of part (I) is carried out in Section~\ref{sec:slowload}, while part (II) is proved in Section~\ref{sec:discrete}. The rest of this section is devoted to the main properties of the inertial cost.
	
	\begin{prop}\label{propmin}
		Fix $t\in [0,T]$, $u_1,u_2\in X$ and $p\in RCP_\V$. 
		Then the inertial energy-dissipation cost related to $p$ can be computed as follows:
		\begin{equation}\label{limminN}
			c^{\mathbb{M},p}_t(u_1,u_2)=\lim\limits_{N\to +\infty}\min_{v\in V^{\mathbb{M},N}_{u_1,u_2}}\int_{-N}^{N} p(\dot{v}(r),-\mathbb{M}\ddot{v}(r)-D_x\mc E(t,v(r)))
			\,\mathrm{d}r.
		\end{equation}
	\end{prop}
	\begin{proof}
		For a fixed $N\in \N$, let $\{v_j\}_{j\in\N}\subseteq V^{\mathbb{M},N}_{u_1,u_2}$ be an infimizing sequence for
		\begin{equation}\label{infN}
			\inf_{v\in V^{\mathbb{M},N}_{u_1,u_2}}\int_{-N}^{N} p(\dot{v}(r),-\mathbb{M}\ddot{v}(r)-D_x\mc E(t,v(r)))
			\,\mathrm{d}r\,.
		\end{equation}
		
		By the definition of $V^{\mathbb{M},N}_{u_1,u_2}$, especially from the bound on the second derivative, we deduce that, up to a not relabelled subsequence, it holds
		\begin{equation*}
			v_j\wto v\quad\text{ weakly in }W^{2,2}(-N,N;X),\quad\text{for some }v\in V^{\mathbb{M},N}_{u_1,u_2}.
		\end{equation*}
		Fo the sake of clarity we introduce the following notation:
		\begin{equation*}
			w_j:=-\M\ddot{v}_j-D_x\mc E(t,v_j),\quad w:=-\M\ddot{v}-D_x\mc E(t,v),
		\end{equation*}
		and we notice that 
		\begin{equation*}
			w_j\wto w,\quad\text{weakly in }L^2(-N,N;X^*).
		\end{equation*}
		By $(ii)$ and $(iv)$ of Definition~\ref{RCP}, we observe that the map $w\mapsto \int_{-N}^{N} p(\dot{v}(r),w(r))\,\mathrm{d}r$ is convex and strongly continuous in $L^2(-N,N;X^*)$, and thus weakly lower semicontinuous. Hence we get
		\begin{align*}
			&\quad\int_{-N}^{N} p(\dot{v}(r),w(r))\,\mathrm{d}r\le \liminf_{j\to +\infty}\int_{-N}^{N} p(\dot{v}(r),w_j(r))\,\mathrm{d}r\\
			&\le \liminf_{j\to +\infty}\int_{-N}^{N} p(\dot{v}_j(r),w_j(r))\,\mathrm{d}r+\limsup_{j\to +\infty}\int_{-N}^{N}|p(\dot{v}_j(r),w_j(r))-p(\dot{v}(r),w_j(r))|\,\mathrm{d}r.
		\end{align*}
		Since $\{v_j\}_{j\in\N}$ is an infimizing sequence, we conclude that the minimum in \eqref{infN} is attained if we prove that the last term in the above estimate vanishes. To this aim we first notice that for almost every $r\in [-N,N]$ the sequence $(\dot{v}_j(r),w_j(r))$ is  contained in a compact subset $ K$ of $X\times X^*$; then let $\omega$ be a modulus of continuity of $p$ in $ K$. We thus obtain
		\begin{align*}
			\limsup_{j\to +\infty}\int_{-N}^{N}|p(\dot{v}_j(r),w_j(r))-p(\dot{v}(r),w_j(r))|\,\mathrm{d}r&\le \limsup_{j\to +\infty}\int_{-N}^{N}\omega(\norm{\dot{v}_j(r)-\dot{v}(r)})\d r,
		\end{align*}
		which vanishes since $\dot{v}_j\wto \dot{v}$ weakly in $W^{1,2}(-N,N;X)$, and thus strongly in $C^0([-N,N];X)$.\\
		To obtain formula \eqref{limminN} we simply notice that the map $$N\mapsto \min_{v\in V^{\mathbb{M},N}_{u_1,u_2}}\int_{-N}^{N} p(\dot{v}(r),-\mathbb{M}\ddot{v}(r)-D_x\mc E(t,v(r)))
		\,\mathrm{d}r,$$
		is nonincreasing. Indeed if $N\le M$, any minimizer $v_N$ in $[-N,N]$ can be trivially extended constant to $[-M,M]$, thus obtaining a competitor in $V^{\mathbb{M},M}_{u_1,u_2}$ (we recall that $p(0,w)=0$ for every $w\in X^*$). Hence \eqref{limminN} follows by the very definition of the inertial energy-dissipation cost \eqref{eq:cost}. 
	\end{proof}
	
	With the following Proposition, we prove that the cost function $c^{\mathbb{M},p}_t$ is a (possibly asymmetric) distance. We point out that in the vanishing-viscosity setting of \cite{MielkRosSav12} this distance is induced by a Finsler metric $F(u,\dot{u})$; in our case, the presence of inertia destroys this additional structure.
	
	\begin{prop}\label{propertiescost}
		For every $t\in[0,T]$, $u_1,u_2,u_3\in X$ and $p\in RCP_\V$ we have:
		\begin{itemize}
			\item[(1)] $c^{\mathbb{M},p}_t(u_1,u_2)=0$ if and only if $u_1=u_2$;
			\item[(2)] $c^{\mathbb{M},p}_t(u_1,u_2)\le c^{\mathbb{M},p}_t(u_1,u_3)+c^{\mathbb{M},p}_t(u_3,u_2)$.
		\end{itemize}
		If in addition $p(\cdot,w)$ is symmetric for every $w\in X^*$ (see for instance Proposition~\ref{propYosida}), then the cost is symmetric, i.e.:
		\begin{itemize}
			\item[(3)] $c^{\mathbb{M},p}_t(u_1,u_2)=c^{\mathbb{M},p}_t(u_2,u_1)$.
		\end{itemize}
	\end{prop}
	
	\begin{proof}
		We start proving $(1)$. If $u_1=u_2$, then the constant function is an admissible competitor. Thus, $c^{\mathbb{M},p}_t(u_1,u_2)=0$ since $p(0,w)=0$. On the other hand, if $u_1\neq u_2$, then for every $N\in \N$ and $v\in V^{\mathbb{M},N}_{u_1,u_2}$ by exploiting $(iii)$ in Definition~\ref{RCP}, Jensen's inequality and the one-homogeneity of $\RR$ we have
		\begin{align*}
			\int_{-N}^{N} p(\dot{v}(r),-\mathbb{M}\ddot{v}(r)-D_x\mc E(t,v(r)))
			\,\mathrm{d}r\ge 	\int_{-N}^{N} \RR(\dot v(r))\d r\ge \mc R\left(\int_{-N}^{N}\dot v(r)\d r\right)=\mc R(u_2-u_1)\,.
		\end{align*}
		Hence we get
		\begin{equation*}
			c^{\mathbb{M},p}_t(u_1,u_2)\ge R(u_2-u_1)>0,
		\end{equation*}
		and $(1)$ is proved.\par 
		To show the validity of $(2)$ we fix $N_1,N_2\in \N$ and $v_1\in V^{\mathbb{M},N_1}_{u_1,u_3}$, $v_2\in V^{\mathbb{M},N_2}_{u_3,u_2}$. It is then easy to see that the concatenation
		\begin{equation*}
			v_3(s):=\begin{cases}
				v_1(s+N_2),&\text{if }s\in [-N_1-N_2, N_1-N_2],\\
				v_2(s-N_1),&\text{if }s\in (N_1-N_2,N_1+N_2],
			\end{cases}
		\end{equation*}
		belongs to $ V^{\mathbb{M},N_1+N_2}_{u_1,u_2}$, and that there holds
		\begin{align*}
			&\quad\int_{-N_1-N_2}^{N_1+N_2} p(\dot{v}_3(s),-\mathbb{M}\ddot{v}_3(s)-D_x\mc E(t,v_3(s)))
			\,\mathrm{d}s\\
			&=	\int_{-N_1}^{N_1} p(\dot{v}_1(r),-\mathbb{M}\ddot{v}_1(r)-D_x\mc E(t,v_1(r)))
			\,\mathrm{d}r+\int_{-N_2}^{N_2} p(\dot{v}_2(r),-\mathbb{M}\ddot{v}_2(r)-D_x\mc E(t,v_2(r)))
			\,\mathrm{d}r.
		\end{align*}
		This yields $(2)$.\par 
		To prove $(3)$ we fix $N\in \N$ and $v\in V^{\mathbb{M},N}_{u_2,u_1}$, and we observe that the backward function $\check{v}(r):=v(-r)$ belongs to $V^{\mathbb{M},N}_{u_1,u_2}$. Then, by exploiting the symmetry of $p(\cdot,w)$, we get
		\begin{align*}
			c^{\mathbb{M},p}_t(u_1,u_2)&\le \int_{-N}^{N} \!p(\dot{\check{v}}(r),-\mathbb{M}\ddot{\check{v}}(r)-D_x\mc E(t,\check{v}(r)))
			\mathrm{d}r=\int_{-N}^{N}\! p(-\dot{{v}}(r),-\mathbb{M}\ddot{{v}}(r)-D_x\mc E(t,{v}(r)))
			\mathrm{d}r\\
			&=\int_{-N}^{N} p(\dot{{v}}(r),-\mathbb{M}\ddot{{v}}(r)-D_x\mc E(t,{v}(r)))
			\mathrm{d}r.
		\end{align*}
		This implies $c^{\mathbb{M},p}_t(u_1,u_2)\le c^{\mathbb{M},p}_t(u_2,u_1)$, and by interchanging the role of $u_1$ and $u_2$ we conclude.
	\end{proof}

	\section{Continuous slow-loading limit}\label{sec:slowload}
	The aim of this section is to show that Inertial Balanced (and Virtual) Viscosity solutions can be obtained as slow-loading limit (i.e. as $\eps\to 0$) of dynamical solutions $\xeps$ to \eqref{mainprob}. Namely, we prove part (I) of Theorem~\ref{mainthm}. Hence, here we are assuming \ref{hyp:E1}-\ref{hyp:E4} and \ref{hyp:R1}.
	
	\subsection{Known results}We first briefly recall the known results proved in \cite{GidRiv}. In particular, for the existence of solutions to \eqref{mainprob} we refer to \cite[Theorem~3.8]{GidRiv}, where the problem is considered under more general assumptions.
	\begin{thm}
		For every pair of initial data $(u^\eps_0,u_1^\eps)\in X\times X$ there exists at least a solution $\xeps\in W^{2,\infty}(0,T;X)$ to the differential inclusion \eqref{mainprob}. Moreover, the following energy identity holds:
		\begin{equation}
			\begin{split}
				&\frac{\varepsilon^2}{2}\|\xepsd(t)\|_{\mathbb{M}}^2+ \mc E(t,\xeps(t))+\int_s^t {\mc R}(\xepsd(r))\,\mathrm{d}r  + \varepsilon\int_s^t |\xepsd(r)|_{\mathbb{V}}^2\,\mathrm{d}r \\
				& =\frac{\varepsilon^2}{2}\|\xepsd(s)\|_{\mathbb{M}}^2 +  \mc E(s,\xeps(s)) +  \int_s^t \partial_t\mc E(r,\xeps(r))\,\mathrm{d}r,
			\end{split}
			\label{eq:energyident1}
		\end{equation}
		for every $0\leq s\leq t\leq T$.
	\end{thm}
	
	\begin{prop}
		Let $\xeps$ be a solution to problem \eqref{mainprob}, and assume $u^\eps_0$, $\eps u^\eps_1$ to be uniformly bounded. Then there exists a positive constant $\overline{C}>0$ such that for every $\varepsilon>0$ the following a priori bounds hold:
		\begin{enumerate}
			\item[\rm (a)] $\max\limits_{t\in [0,T]}\|\xeps(t)\|\leq \overline{C}$;
			\item[\rm (b)] $\max\limits_{t\in [0,T]}\varepsilon\|\xepsd(t)\|_\M\leq \overline{C}$;
			\item[\rm (c)] $\essup\limits_{t\in [0,T]}\varepsilon^2\|\M\xepsdd(t)\|_*\leq \overline{C}$;
			\item[\rm (d)] $\displaystyle\int_0^T\mc R(\xepsd(r))\,\mathrm{d}r\leq \overline{C}$.
		\end{enumerate}
		\label{propbounds}
	\end{prop}
	
	\proof
	See \cite[Corollary~3.4]{GidRiv} for (a), (b) and (d). Assertion (c) follows by exploiting the differential inclusion solved by $\xeps$ together with \eqref{boundedness}.
	\endproof
	
	With the a priori bounds of Proposition~\ref{propbounds} at hand, an argument based on Helly's Selection Theorem provides the existence of a convergent subsequence of dynamic solutions $\xeps$. That is the content of the following proposition, whose proof is given in \cite[Theorem~6.1]{GidRiv} and thus is here omitted.
	
	\begin{prop}\label{thmlimitu}
		Let $u_0^\varepsilon$ and $\varepsilon u_1^\varepsilon$ be uniformly bounded. Then for every sequence $\epsj\to 0$ there exists a subsequence (not relabelled)
		and a function $u\in BV_{\mathcal{R}}([0;T];X)$ such that:
		\begin{enumerate}
			\item[\rm (i)] $u^{\varepsilon_j}(t)\to u(t)$, for every $t\in[0,T]$;
			\item[\rm (ii)] $\displaystyle V_{\mathcal{R}}(u;s,t)\leq \mathop{\lim\inf}_{j\to+\infty}\int_s^t\mathcal{R}(\dot{u}^{\varepsilon_j}(r))\,\mathrm{d}r$, for every $0\leq s\le t\leq T$;
			\item[\rm (iii)] $\varepsilon_j\|\dot{u}^{\varepsilon_j}(t)\|_{\mathbb{M}}\to0$ for every $t\in(0,T]\backslash J_u$, where $J_u$ is the jump set of $u$. 
		\end{enumerate}
	\end{prop}
	
	In addition, arguing as for \cite[Propositions~6.2-6.3]{GidRiv}, it can be proven that the limit evolution $u$ above complies with the local stability condition \eqref{locstab} and a suitable energy inequality. We highlight that a function fulfilling such properties is usually called an \emph{a.e. local solution} to the rate-independent system \eqref{quasistprob} (see for instance \cite[Chapter 3]{MielkRoubbook}).
	
	\begin{prop}\label{prop:stabilityandineq}
		Let $u_0^\varepsilon$ and $\varepsilon u_1^\varepsilon$ be uniformly bounded, and $u$ be as in Proposition~\ref{thmlimitu}. Then the inequality
		\begin{equation*}
			\int_s^t \mathcal{R}(v)+\langle D_x\mathcal{E}(r, u(r)),v\rangle\,\mathrm{d}r\geq0,
		\end{equation*}
		holds for every $v\in X$ and for every $0\leq s\leq t\leq T$. In particular, the left and right limits of $u$ are locally stable; i.e., they fulfill the inequalities
		\begin{align}
			&\mathcal{R}(v)+\langle D_x\mathcal{E}(t, u^-(t)),v\rangle\geq 0\,,\quad \mbox{ for every $v\in X$ and for every $t\in(0,T]$;}\label{LS-}\\
			&\mathcal{R}(v)+\langle D_x\mathcal{E}(t, u^+(t)),v\rangle\geq 0\,,\quad \mbox{ for every $v\in X$ and for every $t\in[0,T]$,}\label{LS+}
		\end{align}
		or equivalently
		\begin{align*}
			&-D_x\mc E(t,u^-(t))\in  K^*,\quad\mbox{ for every }t\in (0,T];\\
			&-D_x\mc E(t,u^+(t))\in  K^*,\quad\mbox{ for every }t\in [0,T].
		\end{align*}
		Moreover, the energy inequality
		\begin{equation}
			\mathcal{E}(t,u^+(t))+ V_{\mathcal{R}}(u;s-,t+)\leq \mathcal{E}(s,u^-(s)) + \int_s^t\partial_t \mathcal{E}(r,u(r))\,\mathrm{d}r,
			\label{LEE}
		\end{equation}
		holds for every $0<s\leq t\leq T$. If in addition $\varepsilon u_1^\varepsilon\to0$, then \eqref{LEE} holds true also for $s=0$.
	\end{prop}
	
	It is worth mentioning that, under the additional assumption of (uniform) convexity on the energy $\mc E$, in \cite{GidRiv} the authors were able to deduce that the limit function $u$ is continuous and that \eqref{LEE} is actually an energy equality. They thus obtained in the limit an Energetic solution of the rate-independent problem.
	\subsection{Characterization of the energy loss at jumps}
	In the nonconvex setting continuity of the limit function is no more reasonable, hence the gap of the energy in \eqref{LEE} has to be characterized. This first proposition shows that, as expected, the peculiar behaviour of the limit function $u$ is restricted to its (essential) jump set.
	
	\begin{prop}\label{propatomic}
		Let $u_0^\varepsilon\to u_0$, $\eps u_1^\eps\to 0$, and $u$ be as in Proposition~\ref{thmlimitu}. Then there exists a positive Radon measure $\mu$ such that for every $0\le s\le t\le T$ there holds
		\begin{equation*}
			\mc E(t,u^+(t))+V_\RR(u_{\rm co};s,t)+\sum_{r\in J^{\rm e}_u\cap[s,t]}\mu (\{r\})=	\mc E(s,u^-(s))+\int_{s}^{t}\partial_t\mc E(r,u(r))\d r.
		\end{equation*}
		In particular, $\mc E(t,u^-(t))-\mc E(t,u^+(t))=\mu (\{t\})\ge 0$ for every $t\in J^{\rm e}_u$.
	\end{prop}
	\begin{proof}
		By reasoning as in \cite[Theorem 5.4]{ScilSol} it is easy to see that the map $t\mapsto \mc E(t,u^+(t))-\int_{0}^{t}\partial_t\mc E(r,u(r))\d r$ is nonincreasing; it essentially follows from the energy balance \eqref{eq:energyident1} by dropping the dissipated energy (i.e. the terms with $\RR(\cdot)$ and $|{\cdot}|_\V$) and controlling the kinetic energy in the limit $\eps\to 0$ by means of (iii) in Proposition~\ref{thmlimitu}. This implies the existence of a positive Radon measure ${\mu}$ for which
		\begin{equation}\label{mu}
			\mc E(t,u^+(t))+{\mu}([s,t])=	\mc E(s,u^-(s))+\int_{s}^{t}\partial_t\mc E(r,u(r))\d r,\,\text{ for every }0\le s\le t\le T.
		\end{equation}
		This in particular yields that the distributional derivative of $t\mapsto \mc E(t,u(t))$, denoted by $\mc E(\cdot,u(\cdot))'$, fulfils the relation
		\begin{equation}\label{distrder}
			\mc E(\cdot,u(\cdot))'=-{\mu}+\partial_t\mc E(\cdot,u(\cdot))\mc{L}^1.
		\end{equation}
		On the other hand, by the chain-rule formula in BV (see for instance \cite[Theorem 4.1]{CrDeCicco}, and recall Remark~\ref{rmkE}), it holds
		\begin{equation}\label{chainrule}
			\begin{aligned}
				\mc E(\cdot,u(\cdot))'&=\partial_t\mc E(\cdot,u(\cdot))\mc{L}^1+\scal{D_x\mc E(\cdot,u(\cdot))}{\dot{u}(\cdot)}\mc L^1\\
				&+\scal{D_x\mc E(\cdot,u(\cdot))}{\frac{\d u'_{\rm Ca}}{\d\lambda}(\cdot)}\lambda+\left[\mc E(\cdot,u^+(\cdot))-\mc E(\cdot,u^-(\cdot))\right]\mc H^0\llcorner J^{\rm e}_u,
			\end{aligned}
		\end{equation}
		where $\lambda=\mc L^1+|u'_{\rm Ca}|$ and we recall $u'=u'_{\rm co}+u'_{\rm J}=\dot{u}\mc L^1+u'_{\rm Ca}+u'_{\rm J}$ (see \eqref{eq:decompositionm}).\\
		By combining \eqref{distrder} and \eqref{chainrule} we obtain
		\begin{align*}
			{\mu}&=-\scal{D_x\mc E(\cdot,u(\cdot))}{\dot{u}(\cdot)\frac{\d \mc L^1}{\d\lambda}(\cdot)+\frac{\d u'_{\rm Ca}}{\d\lambda}(\cdot)}\lambda-\left[\mc E(\cdot,u^+(\cdot))-\mc E(\cdot,u^-(\cdot))\right]\mc H^0\llcorner J^{\rm e}_u\\
			&=-\scal{D_x\mc E(\cdot,u(\cdot))}{\frac{\d u'_{\rm co}}{\d\lambda}(\cdot)}\lambda-\left[\mc E(\cdot,u^+(\cdot))-\mc E(\cdot,u^-(\cdot))\right]\mc H^0\llcorner J^{\rm e}_u.
		\end{align*}
		Last equality yields
		\begin{equation}\label{mudiff}
			\frac{\d {\mu}_{\rm co}}{\d\lambda}(t)=-\scal{D_x\mc E(t,u(t))}{\frac{\d u'_{\rm co}}{\d\lambda}(t)},\quad \text{ for }\lambda{-}a.e.\,t\in[0,T],
		\end{equation}
		where we define ${\mu}_{\rm co}:={\mu}+\left[\mc E(\cdot,u^+(\cdot))-\mc E(\cdot,u^-(\cdot))\right]\mc H^0\llcorner J^{\rm e}_u$.\\
		By choosing $v=\frac{\d u'_{\rm co}}{\d\lambda}(t)$ in the local stability condition \eqref{LS+}, and using \eqref{mudiff} we deduce that 
		\begin{equation*}
			\RR\left(\frac{\d u'_{\rm co}}{\d\lambda}(t)\right)\ge-\scal{D_x\mc E(t,u(t))}{\frac{\d u'_{\rm co}}{\d\lambda}(t)}=\frac{\d {\mu}_{\rm co}}{\d\lambda}(t),\quad \text{ for }\lambda{-}a.e.\,t\in[0,T].
		\end{equation*}
		By integrating the above inequality in $[s,t]\subseteq [0,T]$, recalling \eqref{contvar}, we finally get
		\begin{equation}\label{ineq1}
			V_\RR(u_{\rm co};s,t)=\int_{s}^{t}\RR\left(\frac{\d u'_{\rm co}}{\d\lambda}(r)\right)\d \lambda(r)\ge \int_{s}^{t}\frac{\d {\mu}_{\rm co}}{\d\lambda}(r)\d \lambda(r)={\mu}([s,t]\setminus J^{\rm e}_u).
		\end{equation}
		To obtain the reverse inequality we combine \eqref{mu} and \eqref{LEE}, and use \eqref{eq: representation} and \eqref{essvar} to get
		\begin{align*}
			\mu([s,t])=\mc E(s,u^-(s))-\mc E(t,u^+(t))+\int_{s}^{t}\partial_t\mc E(r,u(r))\d r\ge V_\mc R(u;s-,t+)\ge \RR(u')([s,t]).	
		\end{align*}
		Since both $\mu$ and $\mc R(u')$ are Radon measures, the above inequality implies
		\begin{equation*}
			\mu(B)\ge \RR(u')(B),\quad\text{for every Borel set }B\subseteq[0,T].
		\end{equation*}
		In particular we deduce
		\begin{equation}\label{ineq2}
			\mu([s,t]\setminus J^{\rm e}_u)\ge \RR(u')([s,t]\setminus J^{\rm e}_u)=\int_{s}^{t}	\RR\left(\frac{\d u'_{\rm co}}{\d\lambda}(r)\right)\d \lambda(r)=V_\RR(u_{\rm co};s,t).
		\end{equation}
		By joining \eqref{ineq1} with \eqref{ineq2} we finally obtain
		\begin{equation*}
			\mu([s,t])=\mu ([s,t]\setminus J^{\rm e}_u)+\mu(J^{\rm e}_u\cap[s,t])=V_\RR(u_{\rm co};s,t)+\sum_{r\in J^{\rm e}_u\cap[s,t]}\mu(\{r\}),
		\end{equation*}
		and we conclude.
	\end{proof}
	
	Thanks to Proposition~\ref{thmupperbound} we already know that the inertial cost $c^{\mathbb{M},p}_t(u^-(t),u^+(t))$ is an upper bound bound for $\mu(\{t\})$ for every $p\in RCP_\V$. We now prove that it is a lower bound as well, thus concluding the proof of part (I) of Theorem~\ref{mainthm}.
	
	\begin{prop}
		Let $u_0^\varepsilon\to u_0$, $\eps u_1^\eps\to 0$, and $u$ be as in Proposition~\ref{thmlimitu}. Then for every $t\in [0,T]$ it holds
		\begin{equation}
			\mc E(t,u^-(t))-\mc E(t,u^+(t))\geq \sup_{p\in RCP_\V}c^{\mathbb{M},p}_t(u^-(t),u^+(t))\,.
			\label{lowbounddiss}
		\end{equation}
		\label{proposizione1}
	\end{prop}
	
	\proof 
	Let $u^{\varepsilon_j}$ be the subsequence obtained in Proposition~\ref{thmlimitu}. We restrict to the case $t\in J^{\rm e}_u$, since for any $t\in [0,T]\backslash J^{\rm e}_u$ inequality \eqref{lowbounddiss} holds as a trivial equality in view of $(1)$ in Proposition~\ref{propertiescost}. If $t=0$ we  convene that the function $u^{\varepsilon_j}$ is extended to a left neighborhood of $0$ with an affine function of constant slope $u^{\varepsilon_j}_1$. Reasoning as in \cite[Proposition 5.8]{ScilSol}, by a diagonal argument we can find sequences $t_j^-\nearrow t$ and $t_j^+\searrow t$ and a (further) subsequence, still denoted by $\varepsilon_j$, such that
	\begin{align}
		&u^{\varepsilon_j}(t_j^-)\to u^-(t),\quad u^{\varepsilon_j}(t_j^+)\to u^+(t),\label{hyp1}\\
		&\varepsilon_j\dot{u}^{\varepsilon_j}(t_j^-)\to 0,\quad \varepsilon_j\dot{u}^{\varepsilon_j}(t_j^+)\to 0\,,\label{hyp2}
	\end{align}
	as $j\to+\infty$.\\
	By exploiting \eqref{augmeb} and from the definition of the contact potential $p_\V$ we thus infer:
	\begin{equation*}
		\begin{split}
			&\qquad\mc E(t,u^-(t))-\mc E(t,u^+(t))\\
			&=\displaystyle\lim_{j\to+\infty}\left[\mc E(t_j^-,u^{\varepsilon_j}(t_j^-)){-}\mc E(t_j^+,u^{\varepsilon_j}(t_j^+))
			+ \frac{\varepsilon_j^2}{2}\|\dot{u}^{\varepsilon_j}(t_j^-)\|_{\mathbb{M}}^2{-}\frac{\varepsilon_j^2}{2}\|\dot{u}^{\varepsilon_j}(t_j^+)\|_{\mathbb{M}}^2+\int_{t_j^-}^{t_j^+}\!\!\!\partial_t \mc E(r,u^{\varepsilon_j}(r))\,\mathrm{d}r\right]\\
			&=\lim_{j\to+\infty}\int_{t_j^-}^{t_j^+}{\mc R}_{\eps_j}(\dot{u}^{\varepsilon_j}(r))+\RR_{\eps_j}^*(w^{\epsj}(r))\d r\ge \limsup_{j\to +\infty}\int_{t_j^-}^{t_j^+}p_\V(\dot{u}^{\varepsilon_j}(r),w^{\epsj}(r))\d r\,.
		\end{split}
	\end{equation*}
	
	We now take any $p\in RCP_\V$, and from $(iii)$ in Definition~\ref{RCP} we can continue the previous inequality, getting
	\begin{equation}\label{tofreeze}
		\begin{split}
			\mc E(t,u^-(t))-\mc E(t,u^+(t))&\ge\limsup_{j\to +\infty}\int_{t_j^-}^{t_j^+}p(\dot{u}^{\varepsilon_j}(r),w^{\epsj}(r))\d r\\
			&= \limsup_{j\to +\infty}\int_{t_j^-}^{t_j^+}p(\dot{u}^{\varepsilon_j}(r),- \varepsilon_j^2 \mathbb{M}\ddot{u}^{\varepsilon_j}(r) - D_x\mc E(r,{u}^{\varepsilon_j}(r)))\d r.
		\end{split}
	\end{equation}
	
	Then, by using the Lipschitzianity of $p$ in the second variable (property $(iv)$ in Definition~\ref{RCP}), we notice that
	\begin{align*}
		&\int_{t_j^-}^{t_j^+}|p(\dot{u}^{\varepsilon_j}(r),- \varepsilon_j^2\mathbb{M}\ddot{u}^{\varepsilon_j}(r) - D_x\mc E(r,{u}^{\varepsilon_j}(r)))-p(\dot{u}^{\varepsilon_j}(r),- \varepsilon_j^2 \mathbb{M}\ddot{u}^{\varepsilon_j}(r) - D_x\mc E(t,{u}^{\varepsilon_j}(r)))|\d r\\
		&\le L \int_{t_j^-}^{t_j^+}\norm{\dot{u}^{\varepsilon_j}(r)}\norm{D_x\mc E(t,{u}^{\varepsilon_j}(r))-D_x\mc E(r,{u}^{\varepsilon_j}(r))}_*\d r\,.
	\end{align*}
	If we denote by $\omega$ a modulus of continuity for $D_x\mc E$ on $[0,T]\times B_{\overline C}$, where $\overline{C}$ is the constant of Proposition~\ref{propbounds}, we can bound the last term in the above inequality by
	\begin{equation*}
		L\,\omega(\max\{|t_j^+-t|, |t-t_j^-|\})\int_{t_j^-}^{t_j^+}\norm{\dot{u}^{\varepsilon_j}(r)}\d r\,,
	\end{equation*}
	which vanishes as $j\to +\infty$ thanks to the uniform bound (d) in Proposition~\ref{propbounds}.\\
	This means that we can freeze the time $t$ in $D_x\mc E$ of \eqref{tofreeze}, getting
	\begin{equation}\label{freezed}
		\mc E(t,u^-(t))-\mc E(t,u^+(t))\ge \limsup_{j\to +\infty}\int_{t_j^-}^{t_j^+}p(\dot{u}^{\varepsilon_j}(r),- \varepsilon_j^2 \mathbb{M}\ddot{u}^{\varepsilon_j}(r) - D_x\mc E(t,{u}^{\varepsilon_j}(r)))\d r\,.
	\end{equation}	
	Following \cite[Proposition 5.8]{ScilSol}, we now set
	\begin{equation}\label{defvk}
		v_j(\tau):=u^{\varepsilon_j}(\varepsilon_j\tau+t_j^-),\quad \text{ for every }\tau\in[0,\sigma_j]\,,
	\end{equation}
	where we denoted by $\sigma_j$ the ratio $\frac{t_j^{+}-t_j^{-}}{\varepsilon_j}$. Then, through the change of variables $r=\varepsilon_j\tau+t_j^{-}$ and recalling the one-homogeneity of $p$ with respect to the first variable we obtain
	\begin{equation*}
		\begin{split}
			\quad\int_{t_j^-}^{t_j^+}p(\dot{u}^{\varepsilon_j}(r),- \varepsilon_j^2 \mathbb{M}\ddot{u}^{\varepsilon_j}(r) - D_x\mc E(t,{u}^{\varepsilon_j}(r)))\d r= \int_{0}^{\sigma_j}p(\dot{v}_{j}(\tau),-\mathbb{M}\ddot{v}_{j}(\tau)-D_x\mc E(t,v_{j}(\tau)))\,\mathrm{d}\tau\,.
		\end{split}
	\end{equation*}
	
	We also notice that \eqref{hyp1}-\eqref{hyp2} can be re-read for $v_j$ as
	\begin{equation}\label{rewrite}
		\begin{split}
			&\|v_j(0)-u^-(t)\|\to 0\,, \quad \|v_j(\sigma_j)-u^+(t)\|\to 0\,,\\
			&\|\dot{v}_j(0)\|\to 0\,, \quad \|\dot{v}_j(\sigma_j)\|\to 0\,,
		\end{split}
	\end{equation}
	as $j\to +\infty$. \\	
	We now introduce the functions
	\begin{equation*}
		g(x)=3x^2-2x^3\,,\quad h(x)=-x^2(1-x)\,,\quad x\in[0,1]\,,
	\end{equation*} and the competitor
	\begin{equation*}
		\tilde{v}_j(\tau)=
		\begin{cases}
			u^-(t), & \tau\leq-1\,,\\
			u^-(t)+g(\tau+1)(v_j(0)-u^-(t))+h(\tau+1)\dot{v}_j(0)\,, & \tau\in[-1,0]\,,\\
			v_j(\tau), & \tau\in[0,\sigma_j]\,,\\
			v_j(\sigma_j)+g(1+\sigma_j-\tau) (u^+(t)-v_j(\sigma_j))-h(1+\sigma_j-\tau)\dot{v}_j(\sigma_j)\,, & \tau\in[\sigma_j,\sigma_j+1]\,,\\
			u^+(t), & \tau\geq \sigma_j+1\,.
		\end{cases}
	\end{equation*}
	For the sake of clarity we denote by $\alpha_j(\tau)$ and $\beta_j(\tau)$ the expressions of $\tilde{v}_j$ in $[-1,0]$ and $[\sigma_j,\sigma_j+1]$, respectively, and we notice that by \eqref{rewrite} $\alpha_j$ and $\beta_j$ are uniformly bounded and there holds
	\begin{equation}\label{alphabeta}
		\lim\limits_{j\to +\infty}\left(\max_{\tau\in[-1,0]}(\norm{\dot{\alpha}_j(\tau)}+\norm{\ddot{\alpha}_j(\tau)})+\max_{\tau\in[\sigma_j,\sigma_j+1]}\left(\norm{\dot{\beta}_j(\tau)}+\norm{\ddot{\beta}_j(\tau)}\right)\right)=0.
	\end{equation}
	Fix now an arbitrary $N_j \in \mathbb{N}$ with $2N_j-1>\sigma_j+1$, and recalling that $p(0,w)=0$ observe that
	\begin{align*}
		&\quad\,\,\int_{-1}^{2N_j-1}p(\dot{\tilde{v}}_{j}(\tau),-\mathbb{M}\ddot{\tilde{v}}_{j}(\tau)-D_x\mc E(t,\tilde v_{j}(\tau)))\,\mathrm{d}\tau\\
		&=\int_{0}^{\sigma_j}p(\dot{v}_{j}(\tau),-\mathbb{M}\ddot{v}_{j}(\tau)-D_x\mc E(t,v_{j}(\tau)))\,\mathrm{d}\tau\\
		&+\int_{-1}^{0}p(\dot{\alpha}_j(\tau), -\mathbb{M}\ddot{\alpha}_{j}(\tau)-D_x\mc E(t,\alpha_{j}(\tau)) )\d\tau+\int_{\sigma_j}^{\sigma_j+1}p(\dot{\beta}_j(\tau), -\mathbb{M}\ddot{\beta}_{j}(\tau)-D_x\mc E(t,\beta_{j}(\tau)))\d\tau\,.
	\end{align*}
	By means of \eqref{cont0} and  \eqref{alphabeta} it is easy to see that both terms in the last line above vanish as $j\to +\infty$. This allows us to continue \eqref{freezed} getting
	\begin{equation}\label{almostlast}
		\mc E(t,u^-(t))-\mc E(t,u^+(t))\ge \limsup_{j\to +\infty}\int_{-1}^{2N_j-1}p(\dot{\tilde{v}}_{j}(\tau),-\mathbb{M}\ddot{\tilde{v}}_{j}(\tau)-D_x\mc E(t,\tilde v_{j}(\tau)))\,\mathrm{d}\tau\,.
	\end{equation}
	With the time translation $\hat v_j(s)=\tilde{v}_j(s+N_j-1)$, we finally construct a function belonging to $V^{\mathbb{M},N_j}_{u^-(t),u^+(t)}$ (indeed notice that the bound on the second derivative follows from \eqref{alphabeta}, \eqref{defvk} and (c) in Proposition~\ref{propbounds}). From \eqref{almostlast} we thus get
	\begin{align*}
		\mc E(t,u^-(t))-\mc E(t,u^+(t))&\ge \limsup_{j\to +\infty}\int_{-1}^{2N_j-1}p(\dot{\tilde{v}}_{j}(\tau),-\mathbb{M}\ddot{\tilde{v}}_{j}(\tau)-D_x\mc E(t,\tilde v_{j}(\tau)))\,\mathrm{d}\tau\\
		&=\limsup_{j\to +\infty}\int_{-N_j}^{N_j}p(\dot{\hat{v}}_{j}(s),-\mathbb{M}\ddot{\hat{v}}_{j}(s)-D_x\mc E(t,\hat v_{j}(s)))\,\mathrm{d}s\\
		&\ge c^{\mathbb{M},p}_t(u^-(t),u^+(t))\,,
	\end{align*}
	and by the arbitrariness of $p\in RCP_\V$ we conclude.
	
	\endproof

	\section{Incremental minimization scheme}\label{sec:discrete}
	
		This last section is devoted to the proof of part (II) of Theorem~\ref{mainthm}: we namely show that IBV and IVV solutions can be also obtained as a limit of time-discrete solutions when $\eps$ and the time step $\tau$ vanish simultaneously (with a certain rate). For this, in addition to the assumptions of previous section, we need to require \ref{hyp:E3'} and \ref{hyp:E5}.
	
	Let $T>0$ and $\tau\in (0,1)$ be a fixed time step such that $\frac{T}{\tau}\in\mathbb{N}$. We consider the corresponding induced partition $\Pi_\tau:=\{t^k\}_k$ of the time-interval $[0,T]$, defined by $t^k:=k\tau$ where $k=0,1,\dots,T/\tau$. For future use we also define $t^{-1}:=-\tau$ and we set $\mathcal{K}_\tau:=\{1,\dots,T/\tau\}$ and $\mc K_\tau^0:=\mc K_\tau\cup\{0\}$. \par	
	We construct a recursive sequence $\{u^k_{\tau,\varepsilon}\}_{k\in \mc K_\tau}$ by solving the following iterated minimum problem à la Minimizing Movements:
	\begin{subequations}\label{schemeincond}
			\begin{equation}
			u^k_{\tau,\varepsilon}\in \mathop{\rm arg\,min}\limits_{x\in X}\mathcal{F}_{\tau,\varepsilon}(t^k,x, u^{k-1}_{\tau,\varepsilon}, u^{k-2}_{\tau,\varepsilon}),\quad k\in \mc K_\tau,
			\label{scheme}
		\end{equation}
		with initial conditions 
		\begin{equation}
			u^0_{\tau,\varepsilon}:=u_0^\varepsilon\,,\quad u^{-1}_{\tau,\varepsilon}:= u_0^\eps-\tau u_1^\varepsilon\,,
			\label{eq:condin}
		\end{equation}
	\end{subequations}
	where
	\begin{align*}
		\mathcal{F}_{\tau,\varepsilon}(t^k,x, u^{k-1}_{\tau,\varepsilon},u^{k-2}_{\tau,\varepsilon}):=& \frac{\varepsilon^2}{2\tau^2}\|x-2u_{\tau,\varepsilon}^{k-1}+u_{\tau,\varepsilon}^{k-2}\|^2_{\mathbb{M}}+\frac{\eps}{2\tau}|x-u_{\tau,\varepsilon}^{k-1}|^2_{\mathbb{V}}+\mc R\left({x-u_{\tau,\varepsilon}^{k-1}}\right) + \mc E(t^{k},x)\\
		&+\frac{\Lambda_\V}{4}\|x-u_{\tau,\varepsilon}^{k-1}\|^2_\mathbb{I},
	\end{align*}
	and
	\begin{equation}\label{lambdav}
		\Lambda_\mathbb V:=\begin{cases}
			0,&\text{if }\mathbb{V} \text{ is positive-definite}, \\
			\Lambda,&\text{otherwise}.
		\end{cases}
	\end{equation}
	with $\Lambda$ and $\mathbb{I}$ from \ref{hyp:E5}.\par 
	The addition in the functional $\mc F_{\tau,\eps}$ of the last fictitious viscous term, which by definition \eqref{lambdav} of $\Lambda_\V$ is present only if $\V$ is not positive-definite, is needed to deal with the $\Lambda$-convexity assumption \ref{hyp:E5}. If $\V$ is positive-definite and the ratio $\frac \eps\tau$ is the large enough (see \eqref{largeratio}), the second term will be enough to keep $\Lambda$-convexity under control.
	
	We observe that the existence of a minimum in \eqref{scheme} easily follows from the direct method. Furthermore, if $\frac{\eps^2}{\tau^2}$ is large enough (this is the case under the assumption \eqref{eq:vanratio} needed to conclude the whole argument), the minimum is unique by strict convexity of the functional.\par  
	By defining $v_{\tau,\varepsilon}^k:=\frac{u_{\tau,\varepsilon}^k-u_{\tau,\varepsilon}^{k-1}}{\tau}$, we notice that the Euler Lagrange equation solved by $u_{\tau,\varepsilon}^k$ reads as
	\begin{equation}\label{EL}
		\eps^2 \mathbb{M}\frac{v_{\tau,\varepsilon}^k-v^{k-1}_{\tau,\varepsilon}}{\tau}+\eps\mathbb{V} v_{\tau,\varepsilon}^k +\partial\mc R(v_{\tau,\varepsilon}^k)+D_x\mc E(t^k,u_{\tau,\varepsilon}^k)+\frac{\Lambda_\V}{2} \tau \mathbb{I} v_{\tau,\varepsilon}^k\ni 0.
	\end{equation}
	We also observe that by \eqref{eq:condin} one has $v^0_{\tau,\eps}=u_1^\eps$. Thus, in the limit as $\tau\to 0$ with $\eps$ fixed, we formally (but this could be actually made rigorous, see for instance \cite{RosThom}) recover the dynamic problem \eqref{mainprob}.\medskip
 
	In order to enlighten the notation, from now on we will drop the dependence on $\tau,\varepsilon$ in $u^k_{\tau,\varepsilon}$ and $v^k_{\tau,\varepsilon}$, and we will simply write $u^k$, $v^k$.\par 
	As in the continuous counterpart developed in Section~\ref{sec:slowload}, the first step in the analysis consists in finding uniform a priori estimates, which usually follows by combining an energy inequality together with Gr\"onwall's Lemma. In the discrete setting, we employ the following version of the discrete Gr\"onwall's inequality, whose proof can be found for instance in \cite[Appendix A]{Krusebook}.\par 
	We want to stress that here and hencefort we adopt the convention that an empty sum is equal to $0$.	
	
	\begin{lemma}[\textbf{Gr\"onwall}]\label{discrgron}
		Let $\{\gamma^n\}_{n\in\N}$, $\{f^n\}_{n\in\N}$ be two nonnegative sequences, and let $c\ge 0$. If
		\begin{equation*}
			\gamma^n\le c+\sum_{k=1}^{n-1}f^k\gamma^k,\quad\text{for every }n\in\N,
		\end{equation*}
		then one has
		\begin{equation*}
			\gamma^n\le c\, {\rm exp} {\left(\sum_{k=1}^{n-1}f^k\right)},\quad\text{for every }n\in\N.
		\end{equation*}
	\end{lemma}
	
	\begin{prop}\label{prop:inequality0}
		For every $m,n\in\mathcal{K}^0_\tau$ with $m\le n$ the following discrete energy inequality holds true:
		\begin{equation}
			\begin{split}
				&\quad\frac{\varepsilon^2}{2}{\left\|v^n\right\|^2_{\mathbb{M}}} + \mc E(t^{n},u^{n}) + \sum_{k=m+1}^n\tau\mc R(v^k)+\sum_{k=m+1}^n \tau\left(\varepsilon|v^k|_{\mathbb{V}}^2-\frac{\Lambda-\Lambda_\V}{2}\tau\|v^k\|_\mathbb{I}^2\right)\\
				& \leq  \frac{\varepsilon^2}{2}{\|v^m\|^2_{\mathbb{M}}} + \mc E (t^m,u^m) + \sum_{k=m+1}^n \int_{t^{k-1}}^{t^k}\partial_t\mc E(r,u^{k-1})\,\mathrm{d}r \,.
			\end{split}
			\label{eq:eniquality1}
		\end{equation}
		Furthermore, if $u_0^\varepsilon$ and $\varepsilon u_1^\varepsilon$ are uniformly bounded and
		\begin{equation}\label{largeratio}
			\frac \tau\eps\le \frac{2}{\Lambda V I},\quad\text{ if }\V \text{ is positive-definite},
		\end{equation}
		then there exists $C>0$, independent of $\varepsilon$ and $\tau$, such that 
		\begin{equation}
			\frac{\varepsilon^2}{2}{\left\|v^n\right\|^2_{\mathbb{M}}} + \mc E(t^{n},u^{n}) + \sum_{k=1}^n \tau\mc R(v^k)\leq C,
			\label{eq:bound0}
		\end{equation}
		for every $n\in\mc K^0_\tau$.
		
	\end{prop}
	
	\proof
	
	Testing \eqref{EL} by $\tau v^k$, from \eqref{2.2mielke}, \eqref{eq:assumption} and from the fact that
	\begin{equation*}
		\frac{\norm{x}^2_\mathbb{M}}{2}-\frac{\norm{y}^2_\mathbb{M}}{2}=\frac{\norm{x-y}^2_\mathbb{M}}{2}-\scal{\mathbb{M}(y-x)}{y}\ge -\scal{\mathbb{M}(y-x)}{y},
	\end{equation*}
	we deduce
	\begin{equation*}
		\begin{split}
			\tau\mc R(v^k) =  &- \left \langle \varepsilon^2\mathbb{M}\frac{v^k-v^{k-1}}{\tau}+\varepsilon\mathbb{V}v^k+D_x \mc E(t^{k},u^{k})+\frac{\Lambda_\V}{2}\mathbb{I}\tau v^k, \tau v^k\right\rangle \\
			= & - \varepsilon  \tau|v^k|_{\mathbb{V}}^2  - \varepsilon^2\left \langle \mathbb{M}(v^k-v^{k-1}), v^k\right\rangle +\left\langle D_x \mc E(t^{k},u^{k}), u^{k-1}-u^k\right\rangle- \frac{\Lambda_\V}{2}\|u^k-u^{k-1}\|^2_{\mathbb{I}}\\
			\leq &   - \varepsilon  \tau|v^k|_{\mathbb{V}}^2 + \frac{\varepsilon^2}{2}{\|v^{k-1}\|^2_{\mathbb{M}}} -  \frac{\varepsilon^2}{2}{\|v^k\|^2_{\mathbb{M}}} -\mc E(t^k,u^k) + \mc E (t^k,u^{k-1}) + \frac{\Lambda-\Lambda_\V}{2}\tau^2\|v^k\|_\mathbb{I}^2.
		\end{split}
	\end{equation*}
	Subtracting $\mc E (t^{k-1},u^{k-1})$ from both sides, rearranging the terms and summing upon $k=m,\dots,n$, we obtain \eqref{eq:eniquality1}. \\	
	We now come to the proof of \eqref{eq:bound0}. We first notice that, defining 
	\begin{equation*}
		\begin{cases}\displaystyle
			\gamma^n:= \frac{\varepsilon^2}{2}{\left\|v^n\right\|^2_{\mathbb{M}}}+\mc E(t^{n},u^{n})+\sum_{k=1}^n\tau\mc R(v^k)
			+a_1,\text{ if }n\in\mc K_\tau,\\\displaystyle 
			\gamma^0:= \frac{\varepsilon^2}{2}{\left\|u_1^\eps\right\|^2_{\mathbb{M}}} + \mc E(0,u_0^\eps)+a_1,
		\end{cases}
	\end{equation*}
	where $a_1$ is the constant appearing in \ref{hyp:E3'}, from \eqref{largeratio} it holds 	 
	\begin{equation}\label{start}
		\gamma^n\le \gamma^0+\sum_{k=1}^{n}\int_{t^{k-1}}^{t^k}\partial_t \mc E(r, u^{k-1})\d r,\quad\text{for every }n\in \mc K_\tau.
	\end{equation}
 We indeed observe that the term $\varepsilon|v^k|_{\mathbb{V}}^2-\frac{\Lambda-\Lambda_\V}{2}\tau\|v^k\|_\mathbb{I}^2$ in \eqref{eq:eniquality1} is nonnegative: if $\V$ is not positive-definite, it reduces to $\eps\tau|v^k|_\V^2$; otherwise we exploit \eqref{largeratio}:
 	\begin{align*}
 		\varepsilon\|v^k\|_{\mathbb{V}}^2-\frac{\Lambda}{2}\tau\|v^k\|_\mathbb{I}^2\ge \left(\frac\eps V-\frac \Lambda 2 I\tau \right)\|v^k\|^2\ge 0.
 	\end{align*}
	Thanks to \eqref{rmkLip} we now have
	\begin{equation*}
		\int_{t^{k-1}}^{t^k}\partial_t \mc E(r, u^{k-1})\d r\le (\mc E(t^{k-1}, u^{k-1})+a_1)\int_{t^{k-1}}^{t^k} b(r){ e}^{\int_{t^{k-1}}^{r}b(s)\d s}\d r\le \left({ e}^{\int_{t^{k-1}}^{t^k}b(r)\d r}-1\right)\gamma^{k-1},
	\end{equation*}
	and thus from \eqref{start} we infer
	\begin{align*}
		\gamma^n\le 
		\gamma^0 { e}^{\int_{0}^{\tau}b(r)\d r}+\sum_{k=1}^{n-1} \left({ e}^{\int_{t^{k}}^{t^{k+1}}b(r)\d r}-1\right)\gamma^{k},\quad\text{for every }n\in \mc K_\tau.
	\end{align*}
	Hence, by means of Lemma~\ref{discrgron} we get
	\begin{equation}\label{discrGron}
		\gamma^n\le \gamma^0 { e}^{\int_{0}^{\tau}b(r)\d r} {\rm exp}\left(\sum_{k=1}^{n-1}\left({ e}^{\int_{t^{k}}^{t^{k+1}}b(r)\d r}-1\right)\right),\quad\text{for every }n\in \mc K_\tau.
	\end{equation}
	By defining $B:=\int_{0}^{T}b(r)\d r$ and recalling the elementary inequality
	\begin{equation*}
		{e^x-1}\le \frac{e^B-1}{B}x,\quad\text{for every }x\in [0,B],
	\end{equation*}
	from \eqref{discrGron} we finally obtain
	\begin{equation*}
		\gamma^n\le \gamma^0 e^{e^B-1},\quad\text{for every }n\in \mc K_\tau.
	\end{equation*}
	Since $\gamma^0$ is uniformly bounded by assumption we conclude.

	\endproof

	\begin{cor}\label{cor:equiboundedd}
		Assume $u_0^\varepsilon$ and $\varepsilon u_1^\varepsilon$ are uniformly bounded and \eqref{largeratio}. Then, the following uniform bounds hold for every $n\in \mc K_\tau$:
		\begin{itemize}
			\item[{(j)}] $\|u^n\|\leq C$;
			\item[{(jj)}] $\displaystyle{\varepsilon}\|v^n\|_\mathbb{M}\leq C$;
			\item[{(jjj)}] $\displaystyle{\varepsilon^2}\left\|\mathbb{M}\frac{v^n-v^{n-1}}{\tau}\right\|_*\leq C$;
			\item[{(jjjj)}] $\displaystyle\sum_{k=1}^n\tau\mc R(v^k)\leq C$.
		\end{itemize}
		
	\end{cor}
	
	\proof
	The bounds $(jj)$, $(jjjj)$ can be easily inferred from \eqref{eq:bound0}.  
	We then prove $(j)$. Let $n\in\mc K_\tau$ be fixed; then, with $(jjjj)$, \eqref{Rbounds} and the triangle inequality we have
	\begin{equation*}
		\|u^n\| \leq \|u_0^\varepsilon\|+ \sum_{k=1}^n\|u^k-u^{k-1}\|=\|u_0^\varepsilon\|+ \sum_{k=1}^n\tau\|v^k\| \leq C\,.
	\end{equation*}
	For what concerns $(jjj)$, it can be obtained from the Euler-Lagrange equation \eqref{EL} taking into account \eqref{boundedness}, \ref{hyp:E2} and $(jj)$.
	
	\endproof
	
	\subsection{The main interpolants}
	Once the discrete bounds are obtained, in order to retrieve the continuous framework we need to introduce suitable interpolants of the discrete-in-time sequence $\{u^k\}_{k\in\mc K^0_\tau}$. First, we denote by $\overline{u}_{\tau,\varepsilon}$ (resp., $\underline{u}_{\tau,\varepsilon}$) the left-continuous (resp., right-continuous) piecewise constant interpolant of $\{u^k\}_{k\in\mc K^0_\tau}$, defined by 
	\begin{equation}\label{constantinterpolant}
		\overline{u}_{\tau,\varepsilon}(t):=u^{k},\quad \text{ for } t\in(t^{k-1},t^k],\,\quad \underline{u}_{\tau,\varepsilon}(t):=u^{k-1}\quad \text{ for } t\in[t^{k-1},t^k),\,\, k\in\mathcal{K}^0_\tau\,, 
	\end{equation}
	respectively.  
	We denote by $\widehat{u}_{\tau,\varepsilon}$ the piecewise affine interpolant of $\{u^k\}_{k\in\mc K^0_\tau}$, defined as
	\begin{equation}\label{affineinterpolant}
		\widehat{u}_{\tau,\varepsilon}(t):=\frac{u^k-u^{k-1}}{\tau}(t-t^{k-1})+u^{k-1}= v^k(t-t^{k-1})+u^{k-1},\quad\text{for } t\in(t^{k-1},t^k],\quad k\in\mc K^0_\tau.
	\end{equation}

	Since in the definition of inertial cost \eqref{eq:cost} a second derivative is present, we also need to keep track of its discrete counterpart $\frac{v^k-v^{k-1}}{\tau}$. This is done by finally introducing the function $\widetilde{u}_{\tau,\varepsilon}$ such that $\widetilde{u}_{\tau,\varepsilon}(0)=u^0_\eps$ and whose first derivative is the piecewise affine interpolant of $\{v^k\}_{k\in\mc K^0_\tau}$; namely,
	\begin{subequations}\label{quadraticinterpolant}
		\begin{equation}
			\widetilde{u}_{\tau,\varepsilon}(t):=u^0_\eps+\int_{0}^{t}\dot{\widetilde{u}}_{\tau,\eps}(r)\d r,\quad\quad\text{for }t\in [0,T],
		\end{equation}	
		with
		\begin{equation}
			\dot{\widetilde{u}}_{\tau,\eps}(t):=\frac{v^k-v^{k-1}}{\tau}(t-t^{k-1})+v^{k-1},\quad\text{for } t\in(t^{k-1},t^k],\quad k\in\mc K_\tau.
		\end{equation}
	\end{subequations}
	Notice indeed that $\widetilde{u}_{\tau,\varepsilon}$ is in $W^{2,\infty}(0,T;X)$ with
	\begin{equation*}
		\ddot{\widetilde{u}}_{\tau,\eps}(t)=\frac{v^k-v^{k-1}}{\tau},\quad\text{for } t\in(t^{k-1},t^k),\quad k\in\mc K_\tau.
	\end{equation*}
	Thus, thanks to $(jjj)$ of Corollary~\ref{cor:equiboundedd}, the function $\widetilde{u}_{\tau,\varepsilon}$ is the correct ``discrete'' counterpart of the continuous dynamic solution $\xeps$ to \eqref{mainprob}.
	
	For any $t\in(-\tau,T]$, we also denote by $\mathfrak{t}_\tau$ the least point of partition $\Pi_\tau$ which is greater or equal to $t$; i.e., it is defined as
	\begin{equation}
		\mathfrak{t}_\tau:=\min\{r\in\Pi_\tau:\, r\geq t\}.
		\label{node}
	\end{equation}
	Note that $\mathfrak{t}_\tau\searrow t$ as $\tau\to0$ (for $t\in [0,T]$).	\\
	We finally define a piecewise constant interpolant of the values $\mc E(\cdot,u)$, setting for every $u\in X$
	\begin{equation*}
		\mc E_\tau(t,u):= \mc E(t^k,u),\quad \text{if }t\in(t^{k-1},t^k],\,\,\, k\in \mc K^0_\tau.
	\end{equation*}	
	From assumptions \ref{hyp:E1} and \ref{hyp:E2} we deduce that, in the limit as $\tau\to0$, 
	\begin{equation}
		\mc E_\tau(t,u) \to \mc E(t,u) \quad \mbox{ and }\quad D_x \mc E_\tau(t,u) \to D_x \mc E(t,u)\,,
		\label{eq:energyinterp}
	\end{equation}
	uniformly with respect to $(t,u)\in[0,T]\times \overline{B_R}$.\\	
	In terms of interpolants, the energy inequality \eqref{eq:eniquality1} can be rewritten as 
	\begin{equation*}
		\begin{split}
			&\displaystyle\frac{\varepsilon^2}{2}\|\dot{\widehat{u}}_{\tau,\varepsilon}(t)\|^2_{\mathbb{M}}+\mc E_\tau(t,\overline{u}_{\tau,\varepsilon}(t))+\int_{\mathfrak{s}_\tau}^{\mathfrak{t}_\tau}\mc R(\dot{\widehat{u}}_{\tau,\varepsilon}(r))\,\mathrm{d}r+\int_{\mathfrak{s}_\tau}^{\mathfrak{t}_\tau} \left(\eps|\dot{\widehat{u}}_{\tau,\varepsilon}(r)|_{\mathbb{V}}^2-\frac{\Lambda-\Lambda_\V}{2}\tau\|\dot{\widehat{u}}_{\tau,\varepsilon}(r)\|^2_\mathbb{I}\right)\mathrm{d}r \\
			& \le  \displaystyle\frac{\varepsilon^2}{2}\|\dot{\widehat{u}}_{\tau,\varepsilon}(s)\|^2_{\mathbb{M}}+\mc E_\tau(s,\overline{u}_{\tau,\varepsilon}(s))  + \int_{\mathfrak{s}_\tau}^{\mathfrak{t}_\tau}\partial_t\mc E(r,\underline{u}_{\tau,\varepsilon}(r))\,\mathrm{d}r ,
		\end{split}
	\end{equation*}
	for every $s,t\in(-\tau,T]\backslash\Pi_\tau$ with $s\le t$. Furthermore, Proposition~\ref{prop:inequality0} and the subsequent Corollary~\ref{cor:equiboundedd} can be re-read as follows.
	
	\begin{cor}\label{prop:inequality1}
		Assume that $u_0^\varepsilon$ and $\varepsilon u_1^\varepsilon$ are uniformly bounded and \eqref{largeratio}. Then, there exists $C>0$, independent of $\tau$ and $\varepsilon$, such that 
		\begin{equation*}
			\frac{\varepsilon^2}{2}\|\dot{\widehat{u}}_{\tau,\varepsilon}(t)\|^2_{\mathbb{M}}+\mc E_\tau(t,\overline{u}_{\tau,\varepsilon}(t)) + \int_0^{\mathfrak{t}_\tau}\mc R(\dot{\widehat{u}}_{\tau,\varepsilon}(r))\,\mathrm{d}r \leq C,
		\end{equation*}
		for every $t\in(-\tau,T]\backslash\Pi_\tau$.\par 
		Moreover, up to enlarging the constant $\overline{C}$ appearing in Proposition~\ref{propbounds}, there holds
		\begin{enumerate}
			\item[\rm (a)] $\max\limits_{t\in [0,T]}\|\overline{u}_{\tau,\varepsilon}(t)\|\leq \overline{C}$;
			\item[\rm (b)] $\max\limits_{t\in [0,T]\backslash\Pi_\tau}\varepsilon\|\dot{\widehat{u}}_{\tau,\varepsilon}(t)\|_\M\leq \overline{C}$;
			\item[\rm (c)] $\max\limits_{t\in [0,T]\backslash\Pi_\tau}\varepsilon^2\|\M\ddot{\widetilde{u}}_{\tau,\varepsilon}(t)\|_*\leq \overline{C}$;
			\item[\rm (d)] $\displaystyle\int_0^T\mc R(\dot{\widehat{u}}_{\tau,\varepsilon}(r))\,\mathrm{d}r\leq \overline{C}$.
		\end{enumerate}
	\end{cor}
	
	The next proposition shows that the mismatch between the many interpolants defined above can be bounded by suitable ratios of the parameters $\tau$ and $\eps$. 
	\begin{prop}
		Assume that $u_0^\varepsilon$ and $\varepsilon u_1^\varepsilon$ are uniformly bounded and \eqref{largeratio}. Then we have
		\begin{equation}
			\max\limits_{t\in [0,T]}\big\{\|\overline{u}_{\tau,\varepsilon}(t)-\underline{u}_{\tau,\varepsilon}(t)\|+\|\overline{u}_{\tau,\varepsilon}(t)-\widehat{u}_{\tau,\varepsilon}(t)\|+\|\widetilde{u}_{\tau,\varepsilon}(t)-\widehat{u}_{\tau,\varepsilon}(t)\|\big\}\le C\frac \tau\eps.
			\label{eq:unifconv}
		\end{equation}
		Moreover, it holds
		\begin{equation}\label{l2convtris}
			\max\limits_{t\in [0,T]\backslash\Pi_\tau}\|\dot{\widetilde{u}}_{\tau,\varepsilon}(t)-\dot{\widehat{u}}_{\tau,\varepsilon}(t)\|\leq C\frac{\tau}{\eps^2}\,.
		\end{equation}
		\label{prop:convergence}
	\end{prop}
	\proof
	We first notice that, by virtue of $(jj)$ in Corollary~\ref{cor:equiboundedd} and since $\eps u^\eps_1$ is uniformly bounded, one has
	\begin{equation}
		\max\limits_{k\in\mc K^0_\tau}\|u^k-u^{k-1}\|\leq C{\frac{\tau}{\eps}}\,.
		\label{eq:normadiff}
	\end{equation}
	Thus, let $t\in(t^{k-1},t^k]$ for some $k\in\mc K^0_\tau$. Then there holds
	\begin{equation*}
		\begin{split}
			\|\overline{u}_{\tau,\varepsilon}(t)-\underline{u}_{\tau,\varepsilon}(t)\| & \leq\|u^k-u^{k-1}\|, \\
			\|\overline{u}_{\tau,\varepsilon}(t)-\widehat{u}_{\tau,\varepsilon}(t)\| & = \|u^k-u^{k-1}\|\frac{\tau-(t-t^{k-1})}{\tau}\leq \|u^k-u^{k-1}\|\,.
		\end{split}
	\end{equation*}	
	In order to deal with the last term in \eqref{eq:unifconv} we observe that
	\begin{equation}\label{diffder}
		\dot{\widetilde{u}}_{\tau,\varepsilon}(t)-\dot{\widehat{u}}_{\tau,\varepsilon}(t)=\frac{v^k-v^{k-1}}{\tau}\left({t-t^{k-1}}-\tau\right),\quad \text{for }t\in (t^{k-1},t^k),\, k\in \mc K_\tau,
	\end{equation}
	and thus
	\begin{align*}
		&\quad\,\,\widetilde{u}_{\tau,\varepsilon}(t)-\widehat{u}_{\tau,\varepsilon}(t)=\int_{0}^{t}(\dot{\widetilde{u}}_{\tau,\varepsilon}(r)-\dot{\widehat{u}}_{\tau,\varepsilon}(r))\d r\\
		& =\frac{v^k-v^{k-1}}{\tau}\int_{t^{k-1}}^{t}\left({r-t^{k-1}}-\tau\right)\d r+\sum_{i=1}^{k-1} \frac{v^i-v^{i-1}}{\tau}\int_{t^{i-1}}^{t^i}\left({r-t^{i-1}}-\tau\right)\d r\\
		&=\frac{v^k-v^{k-1}}{2\tau}(t-t^{k-1})\left(t-t^{k-1}-2\tau\right)-\tau\sum_{i=1}^{k-1}\frac{v^i-v^{i-1}}{2}\\
		&=\frac{v^k-v^{k-1}}{2\tau}(t-t^{k-1})\left(t-t^{k-1}-2\tau\right)+\frac{\tau u^\eps_1}{2}-\frac{\tau v^{k-1}}{2}\\
		&=\frac{v^k}{2}\left(\frac{t-t^{k-1}}{\tau}\right)\left(t-t^{k-1}-2\tau\right)-\frac{v^{k-1}}{2}\frac{(t-t^{k-1}-\tau)^2}{\tau}+\frac{\tau u^\eps_1}{2}.
	\end{align*}
	Since $t-t^{k-1}\in (0,\tau)$, we now get
	\begin{align*}
		2\norm{\widetilde{u}_{\tau,\varepsilon}(t)-\widehat{u}_{\tau,\varepsilon}(t)}\le\frac\tau\eps \|\eps u^\eps_1\|+\tau\|{v^k}\|+\tau\|{v^{k-1}}\|=\frac\tau\eps \norm{\eps u^\eps_1}+\|{u^k-u^{k-1}}\|+\|{u^{k-1}-u^{k-2}}\|,
	\end{align*}
	and assertion \eqref{eq:unifconv} follows from \eqref{eq:normadiff}.\par 
	From \eqref{diffder}, the bound \eqref{l2convtris} easily follows by means of $(jjj)$ of Corollary~\ref{cor:equiboundedd}.
	
	\endproof
	
	We are now in a position to prove the analogous of Propositions~\ref{thmlimitu}, \ref{prop:stabilityandineq} and \ref{propatomic} for the sequence of piecewise affine interpolants $\widehat{u}_{\tau,\varepsilon}$. 
	
	\begin{prop}\label{prop:compactnessbis}
		Let $u_0^\varepsilon$ and $\eps u_1^\varepsilon$ be uniformly bounded and assume \eqref{largeratio}. Then for every sequence $(\tau_j,\varepsilon_j)\to(0,0)$ there exists a subsequence (not relabelled) and a function $u\in BV_{\mc R}([0,T];X)$ such that
		\begin{itemize}
			\item[\rm (i)] $\widehat{u}_{\tau_j,\varepsilon_j}(t)\to u(t),$ for every $t\in[0,T]$;
			\item[\rm (ii)] $\displaystyle V_{\mc R}(u;s,t)\leq \mathop{\lim\inf}_{j\to+\infty}\int_{s}^{t}\mc R(\dot{\widehat{u}}_{\tau_j,\varepsilon_j}(r))\,\mathrm{d}r$, for every $0\le s\le t\le T$;
			\item[\rm (iii)] $\varepsilon_j \|\dot{\widehat{u}}_{\tau_j,\varepsilon_j}(t)\|_{\mathbb{M}}\to0$, for a.e. $t\in[0,T]$.
		\end{itemize}
		If in addition $\displaystyle\frac{\tau_j}{\eps_j}\to0$, then it also holds
		\begin{itemize}
			\item[\rm (ii')] $\displaystyle V_{\mc R}(u;s,t)\leq \mathop{\lim\inf}_{j\to+\infty}\int_{\mathfrak{s}_{\tau_j}}^{\mathfrak{t}_{\tau_j}}\mc R(\dot{\widehat{u}}_{\tau_j,\varepsilon_j}(r))\,\mathrm{d}r$, for every $0\le s\le t\le T$;
			\item[\rm (iii')] $\varepsilon_j \|\dot{\widehat{u}}_{\tau_j,\varepsilon_j}(t)\|_{\mathbb{M}}\to0$, for every $t\in(0,T]\backslash(J_u\cup N)$, where $N=\bigcup\limits_{j\in\N}\Pi_{\tau_j}$.
		\end{itemize}
		
	\end{prop}
	\proof
	We can argue as in \cite[Theorem~6.1]{GidRiv}. In view of the a priori bounds of Corollary~\ref{prop:inequality1} and \eqref{Rbounds}, the sequence $\{\widehat{u}_{\tau_j,\varepsilon_j}\}_{j\in\N}$ is uniformly equibounded with uniformly equibounded variation.
	Then, by virtue of the Helly's Selection Theorem there exists a subsequence and a function $u\in BV([0,T];X)$ complying with (i). Furthermore, with \cite[Proposition~4.11 and Lemma~4.12]{GidRiv}, we get $u\in BV_{\mc R}([0,T];X)$ and assertion (ii).\\
	By virtue of Corollary~\ref{prop:inequality1} and \eqref{Rbounds}, we also have
	\begin{equation*}
		\lim\limits_{j\to +\infty}\varepsilon_j \int_0^T \|\dot{\widehat{u}}_{\tau_j,\varepsilon_j}(r)\|\,\mathrm{d}r=0\,,
	\end{equation*}
	whence (iii) follows, up to possibly passing to a further subsequence.\\
	To obtain (ii') it is enough to observe that
	\begin{equation*}
		\int_{s}^{t}\mc R(\dot{\widehat{u}}_{\tau_j,\varepsilon_j}(r))\,\mathrm{d}r=\int_{\mathfrak{s}_{\tau_j}}^{\mathfrak{t}_{\tau_j}}\mc R(\dot{\widehat{u}}_{\tau_j,\varepsilon_j}(r))\,\mathrm{d}r+\int_{s}^{\mathfrak{s}_{\tau_j}}\mc R(\dot{\widehat{u}}_{\tau_j,\varepsilon_j}(r))\,\mathrm{d}r-\int_{t}^{\mathfrak{t}_{\tau_j}}\mc R(\dot{\widehat{u}}_{\tau_j,\varepsilon_j}(r))\,\mathrm{d}r,
	\end{equation*}
	and to notice that the last two terms vanish as $j\to +\infty$ since
	\begin{align*}
		\int_{s}^{\mathfrak{s}_{\tau_j}}\mc R(\dot{\widehat{u}}_{\tau_j,\varepsilon_j}(r))\,\mathrm{d}r\le C \int_{s}^{\mathfrak{s}_{\tau_j}}\|\dot{\widehat{u}}_{\tau_j,\varepsilon_j}(r)\|_\mathbb{M}\d r\le C\frac{\mathfrak{s}_{\tau_j}-s}{\eps_j}\le C\frac{\tau_j}{\eps_j},
	\end{align*}
	and the same holds for the other one.\par 
	The proof of (iii') follows exactly as in \cite[Theorem~6.1]{GidRiv}, by using \eqref{eq:augmebinterp} and recalling that thanks to \eqref{eq:unifconv} we also have $\overline{u}_{\tau_j,\varepsilon_j}(t)\to u(t)$ for every $t\in[0,T]$.
	
	\endproof
	
		\begin{prop}
		Let $u_0^\varepsilon$ and $\varepsilon u_1^\varepsilon$ be uniformly bounded, and let $u$ be the limit function obtained in Proposition~\ref{prop:compactnessbis} from a subsequence satisfying 
		\begin{equation*}
			\lim\limits_{j\to +\infty}\eps_j=\lim\limits_{j\to +\infty}\frac{\tau_j}{\eps_j}=0.
		\end{equation*}
		Then the inequality
		\begin{equation}
			\int_s^t \mathcal{R}(v)+\langle D_x\mathcal{E}(r, u(r)),v\rangle\,\mathrm{d}r\geq0,
			\label{eq:ineqincl}
		\end{equation}
		holds for every $v\in X$ and for every $0\leq s\leq t\leq T$. In particular, the left and right limits of $u$ are locally stable; i.e., they fulfill the inclusions
		\begin{align*}
			&-D_x\mc E(t,u^-(t))\in K^*\,,\quad\mbox{ for every }t\in (0,T]\,;\\
			&-D_x\mc E(t,u^+(t))\in K^*,\quad\mbox{ for every }t\in [0,T]\,.
		\end{align*}
		Moreover, if in addition $\varepsilon u_1^\varepsilon\to0$, there exists a positive Radon measure $\mu$ such that for every $0\le s\le t\le T$ there holds
		\begin{equation*}
			\mc E(t,u^+(t))+V_\RR(u_{\rm co};s,t)+\sum_{r\in J^{\rm e}_u\cap[s,t]}\mu (\{r\})=	\mc E(s,u^-(s))+\int_{s}^{t}\partial_t\mc E(r,u(r))\d r.
		\end{equation*}
		In particular, $\mc E(t,u^-(t))-\mc E(t,u^+(t))=\mu (\{t\})\ge 0$ for every $t\in J^{\rm e}_u$.
	\end{prop}
	
	\proof
	We only prove \eqref{eq:ineqincl}, the remaining assertions being as in \cite[Propositions~6.2-6.3]{GidRiv} and in Proposition~\ref{propatomic}, exploiting (ii') and (iii') of Proposition~\ref{prop:compactnessbis}. From \eqref{EL} and \eqref{2.2mielke}, for every $v\in X$ and $k\in\mathcal{K}_{\tau_j}$ we have
	\begin{align*}
		0\le \mc R(v)+\left \langle \varepsilon_j^2\mathbb{M}\frac{v^{k}-v^{k-1}}{\tau_j}+\varepsilon_j\mathbb{V}v^k+D_x \mc E(t^{k},u^{k})+\frac{\Lambda_\V}{2}\tau_j\mathbb{I}v^k, v\right\rangle.
	\end{align*}
	
	By multiplying both sides by $\tau_j$ and summing on $k$, for every $m,n\in\mathcal{K}^0_{\tau_j}$ with $m\le n$ we now obtain
	\begin{align*}
		0\le \eps_j^2\langle\mathbb{M}(v^n-v^m),v\rangle+\sum_{k=m+1}^n \tau_j\left(\mc R(v)+\langle D_x \mc E(t^{k},u^{k}),v\rangle+\varepsilon_j\langle\mathbb{V}v^k,v\rangle+\frac{\Lambda_\V}{2}\tau_j\langle\mathbb{I}v^k,v\rangle\right),
	\end{align*}
	namely for every $0\le s\le t\le T$ it holds
	\begin{align*}
		0&\le \eps_j^2\left\langle\mathbb{M}(\dot{\widetilde{u}}_{\tau_j,\eps_j}(\mathfrak{t}_{\tau_j})-\dot{\widetilde{u}}_{\tau_j,\eps_j}(\mathfrak{s}_{\tau_j})),v\right\rangle+\int_{\mathfrak{s}_{\tau_j}}^{\mathfrak{t}_{\tau_j}}\mc R(v) + \left \langle D_x \mc E_{\tau_j}(r,\overline{u}_{\tau_j,\varepsilon_j}(r)), v\right\rangle\,\mathrm{d}r\\
		&\quad+\varepsilon_j\int_{\mathfrak{s}_{\tau_j}}^{\mathfrak{t}_{\tau_j}} \left \langle \mathbb{V}\dot{\widehat{u}}_{\tau_j,\varepsilon_j}(r), v\right\rangle\,\mathrm{d}r+\frac{\Lambda_\V}{2}\tau_j \int_{\mathfrak{s}_{\tau_j}}^{\mathfrak{t}_{\tau_j}} \left \langle \mathbb{I}\dot{\widehat{u}}_{\tau_j,\varepsilon_j}(r), v\right\rangle\,\mathrm{d}r.
	\end{align*}	
	Passing to the limit as $j\to+\infty$, by $(ii)$ in Corollary~\ref{cor:equiboundedd} we have that 
	\begin{equation*}
		\lim_{j\to+\infty}\varepsilon_j^2 \left|\left \langle \mathbb{M}\dot{\widetilde{u}}_{\tau_j,\varepsilon_j}(\mathfrak{t}_{\tau_j}), v\right\rangle\right| = \lim_{j\to+\infty} \varepsilon_j^2\left|\left \langle \mathbb{M}\dot{\widetilde{u}}_{\tau_j,\varepsilon_j}(\mathfrak{s}_{\tau_j}), v\right\rangle\right| =0.
	\end{equation*}
	From Corollary~\ref{prop:inequality1}, \eqref{Rbounds}
	and the Cauchy-Schwarz inequality we also have
	\begin{equation*}
		\left|\varepsilon_j\int_{\mathfrak{s}_{\tau_j}}^{\mathfrak{t}_{\tau_j}} \left \langle \mathbb{V}\dot{\widehat{u}}_{\tau_j,\varepsilon_j}(r), v\right\rangle\,\mathrm{d}r \right| \leq C \|v\|\varepsilon_j\int_0^T \|\dot{\widehat{u}}_{\tau_j,\varepsilon_j}(r)\|\,\mathrm{d}r\to0\,,
	\end{equation*}
	and a similar argument shows that the last term in the inequality above vanishes as well as $j\to +\infty$.\par 
	We conclude observing that \eqref{eq:energyinterp} and \eqref{eq:unifconv} allow us to use Dominated Convergence Theorem getting
	\begin{equation*}
		\lim_{j\to+\infty} \int_{\mathfrak{s}_{\tau_j}}^{\mathfrak{t}_{\tau_j}}\mc R(v) + \left \langle D_x \mc E_{\tau_j}(r,\overline{u}_{\tau_j,\varepsilon_j}(r)), v\right\rangle\,\mathrm{d}r = \int_{s}^{t}\mc R(v) + \left \langle D_x \mc E(r,{u}(r)), v\right\rangle\,\mathrm{d}r.
	\end{equation*}
	
	\endproof

	\subsection{The convergence result}
	As already done in the time-continuous setting in \eqref{augmeb}, we now rephrase the energy inequality \eqref{eq:eniquality1} in terms of the De Giorgi's principle. For simplicity, we set:
	\begin{equation}
		w^k:= -\varepsilon^2\mathbb{M}\frac{v^{k}-v^{k-1}}{\tau}-D_x \mc E(t^{k},u^{k})-\frac{\Lambda_\V}{2}\tau\mathbb{I}v^k\,,\quad k\in\mc K_\tau\,,
	\end{equation}
	and, recalling the definitions of interpolants $\overline{u}_{\tau,\varepsilon}$ \eqref{constantinterpolant} and $\widetilde{u}_{\tau,\varepsilon}$ \eqref{quadraticinterpolant}, for $t\in[0,T]\backslash\Pi_\tau$ we define
	\begin{equation*}
		\begin{split}
			\overline{w}_{\tau,\varepsilon}(t)&:= -\varepsilon^2\mathbb{M}\ddot{\widetilde{u}}_{\tau,\varepsilon}(t)-D_x\mc E_\tau(t,\overline{u}_{\tau,\varepsilon}(t))-\frac{\Lambda_\V}{2}\tau\mathbb{I}\dot{\widehat{u}}_{\tau,\eps}(t)\,,\\
			\widetilde{w}_{\tau,\varepsilon}(t)&:= -\varepsilon^2\mathbb{M}\ddot{\widetilde{u}}_{\tau,\varepsilon}(t)-D_x\mc E(t,\widetilde{u}_{\tau,\varepsilon}(t))\,.
		\end{split}
	\end{equation*}
	Then, by virtue of Corollary~\ref{prop:inequality1} and Proposition~\ref{prop:convergence}, if $\displaystyle\frac\tau\eps$ is bounded we deduce
	\begin{equation}
		\max\limits_{t\in [0,T]\backslash\Pi_\tau}\left(\|\overline{w}_{\tau,\varepsilon}(t)\|_*+\|\widetilde{w}_{\tau,\varepsilon}(t)\|_*\right)\leq C\,.
		\label{eq:wbounds}
	\end{equation}
	Furthermore, thanks to the continuity of $D_x\mc E$, if $\displaystyle\frac\tau\eps\to 0$ we also have
	\begin{equation}\label{eq:wvanishing}
		\lim\limits_{(\eps,\tau)\to(0,0)}\max\limits_{t\in [0,T]\backslash\Pi_\tau}\|\overline{w}_{\tau,\varepsilon}(t)-\widetilde{w}_{\tau,\varepsilon}(t)\|_*= 0.
	\end{equation}	
	From \eqref{EL} and recalling that using \eqref{Fenchel} it holds
	\begin{equation*}
		\mc R_\varepsilon(v^k) + \mc R_\varepsilon^*(w^k) =  \langle w^k, v^k\rangle\,,
	\end{equation*}
	inequality \eqref{eq:eniquality1} can be rewritten, arguing in a similar way, as
	
	\begin{equation*}
		\begin{split}
			&\quad\frac{\varepsilon^2}{2}{\left\|v^n\right\|^2_{\mathbb{M}}} + \mc E(t^{n},u^{n}) + \sum_{k=m+1}^n\tau\left(\mc R_\varepsilon(v^k) + \mc R_\varepsilon^*(w^k)\right) \\
			& \leq  \frac{\varepsilon^2}{2}{\|v^m\|^2_{\mathbb{M}}} + \mc E (t^m,u^m) + \sum_{k=m+1}^n \int_{t^{k-1}}^{t^k}\partial_t\mc E(r,u^{k-1})\,\mathrm{d}r +\frac{\Lambda-\Lambda_\V}{2}\tau^2\sum_{k=m+1}^{n}\tau\|v^k\|^2_\mathbb{I},
		\end{split}
	\end{equation*}
	and thus, in terms of interpolants, as
	
	\begin{equation}
		\begin{split}
			&\displaystyle\frac{\varepsilon^2}{2}\|\dot{\widehat{u}}_{\tau,\varepsilon}(t)\|^2_{\mathbb{M}}+\mc E_\tau(t,\overline{u}_{\tau,\varepsilon}(t))+\int_{\mathfrak{s}_\tau}^{\mathfrak{t}_\tau}\mc R_\varepsilon(\dot{\widehat{u}}_{\tau,\varepsilon}(r)) + \mc R^*_\varepsilon(\overline{w}_{\tau,\varepsilon}(r))\,\mathrm{d}r \\
			& \leq \displaystyle\frac{\varepsilon^2}{2}\|\dot{\widehat{u}}_{\tau,\varepsilon}(s)\|^2_{\mathbb{M}}+\mc E_\tau(s,\overline{u}_{\tau,\varepsilon}(s))  + \int_{\mathfrak{s}_\tau}^{\mathfrak{t}_\tau}\partial_t\mc E(r,\underline{u}_{\tau,\varepsilon}(r))\,\mathrm{d}r +\frac{\Lambda-\Lambda_\V}{2}\tau^2\int_{\mathfrak{s}_\tau}^{\mathfrak{t}_\tau}\|\dot{\widehat{u}}_{\tau,\varepsilon}(r)\|_\mathbb{I}^2\d r,
		\end{split}
		\label{eq:augmebinterp}
	\end{equation}
	for every $s,t\in(-\tau,T]\backslash\Pi_\tau$ with $s\le t$.\par 
	To conclude the proof of part (II) of Theorem~\ref{mainthm}, we only have to confirm the validity of an analogous of Proposition~\ref{proposizione1} for the function $u$ obtained with Proposition~\ref{prop:compactnessbis}. For this, we will need to reinforce the assumption $\frac\tau\eps\to0$ by requiring, in addition, that $\frac{\tau}{\varepsilon^2}$ is uniformly bounded (see \eqref{eq:vanratio} below), in order to exploit \eqref{l2convtris}. 
	
	\begin{prop}
		Let $u_0^\varepsilon\to u_0$, $\eps u_1^\eps\to 0$, and let $u$ be the limit function obtained in Proposition~\ref{prop:compactnessbis} from a subsequence satisfying 
		\begin{equation}\label{eq:vanratio}
			\lim\limits_{j\to +\infty}\eps_j=0,\qquad \text{and}\qquad\sup\limits_{j\in\N}\frac{\tau_j}{\varepsilon_j^2}<+\infty.
		\end{equation} 
		Then for every $t\in [0,T]$ it holds
		\begin{equation}
			\mc E(t,u^-(t))-\mc E(t,u^+(t))\geq \sup_{p\in RCP_\V}c^{\mathbb{M},p}_t(u^-(t),u^+(t))\,.
			\label{lowbounddissbis}
		\end{equation}
		\label{proposizione1bis}
	\end{prop}
	
	\proof
	As already remarked in the proof of Proposition~\ref{proposizione1}, it will suffice to prove \eqref{lowbounddissbis} in the case $t\in J^{\rm e}_u$.
	By arguing as in \cite[Proposition~5.9]{SciSol18}, taking into account Proposition~\ref{prop:compactnessbis}, by a diagonal argument we may assume that there are two sequences $t^-_j\nearrow t$ and $t^+_j\searrow t$ such that
	\begin{equation*}
		\lim_{j\to +\infty}\|\widehat{u}_{\tau_j,\varepsilon_j}(t^-_j)- u^-(t)\|+ \|\widehat{u}_{\tau_j,\varepsilon_j}(t^+_j)- u^+(t)\|=0\,,
	\end{equation*}
	and
	\begin{equation}\label{fatto2}
		\lim\limits_{j\to+\infty}\varepsilon_j\dot{\widehat{u}}_{\tau_j,\varepsilon_j}(t^-_j)=	\lim\limits_{j\to+\infty} \varepsilon_j\dot{\widehat{u}}_{\tau_j,\varepsilon_j}(t^+_j)=0\,.
	\end{equation}
	By exploiting \eqref{eq:vanratio}, Proposition~\ref{prop:convergence} yields as a byproduct
	\begin{equation}\label{fatto1}
		\begin{split}
			&\lim_{j\to +\infty}\|\overline{u}_{\tau_j,\varepsilon_j}(t^-_j)- u^-(t)\|+ \|\overline{u}_{\tau_j,\varepsilon_j}(t^+_j)- u^+(t)\|=0\,, \\
			&\lim_{j\to +\infty}\|\widetilde{u}_{\tau_j,\varepsilon_j}(t^-_j)- u^-(t)\|+ \|\widetilde{u}_{\tau_j,\varepsilon_j}(t^+_j)- u^+(t)\|=0\,, \\
		\end{split}
	\end{equation}
	and
	\begin{equation}\label{fatto3}
		\lim\limits_{j\to+\infty}\varepsilon_j\dot{\widetilde{u}}_{\tau_j,\varepsilon_j}(t^-_j)=	\lim\limits_{j\to+\infty} \varepsilon_j\dot{\widetilde{u}}_{\tau_j,\varepsilon_j}(t^+_j)=0\,.
	\end{equation}
	The continuity of $\mc E$ together with \eqref{eq:energyinterp} and \eqref{fatto1} now implies that
	\begin{equation}\label{servira}
		\lim\limits_{j\to+\infty}\mc E_{\tau_j}(t^-_j,\overline{u}_{\tau_j,\varepsilon_j}(t^-_j))-\mc E_{\tau_j}(t^+_j,\overline{u}_{\tau_j,\varepsilon_j}(t^+_j))= \mc E(t,u^-(t))-\mc E(t,u^+(t)).
	\end{equation} 
	 
	For a lighter exposition, with a little abuse of notation we denote by $\mathfrak{t}^-_j$ and $\mathfrak{t}^+_j$ the least points of partition $\Pi_{\tau_j}$ which are greater or equal to $t^-_j$ and $t^+_j$, respectively (see \eqref{node}). By exploiting \eqref{eq:augmebinterp}, \eqref{fatto2}, \eqref{servira},
	and from the definition of the contact potential $p_\V$ we thus infer:
	\begin{equation*}
		\begin{split}
			&\qquad\mc E(t,u^-(t))-\mc E(t,u^+(t))\\
			&=\displaystyle\lim_{j\to+\infty}\biggl[\frac{\varepsilon_j^2}{2}{\left\|\dot{\widehat{u}}_{\tau_j,\varepsilon_j}(t^-_j)\right\|^2_{\mathbb{M}}} + \mc E_{\tau_j}(t^-_j,\overline{u}_{\tau_j,\varepsilon_j}(t^-_j)) + \int_{\mathfrak{t}^-_j}^{\mathfrak{t}^+_j}\partial_t\mc E(r,\underline{u}_{\tau_j,\varepsilon_j}(r))\,\mathrm{d}r \\
			& \qquad\qquad +\frac{\Lambda-\Lambda_\V}{2}\tau^2\int_{\mathfrak{t}_j^-}^{\mathfrak{t}_j^+}\|\dot{\widehat{u}}_{\tau_j,\varepsilon_j}(r)\|_\mathbb{I}^2\d r -\frac{\varepsilon_j^2}{2}{\left\|\dot{\widehat{u}}_{\tau_j,\varepsilon_j}(t^+_j)\right\|^2_{\mathbb{M}}}- \mc E_{\tau_j}(t^+_j,\overline{u}_{\tau_j,\varepsilon_j}(t^+_j)) \biggr]\\
			&\ge\limsup_{j\to +\infty}\int_{\mathfrak{t}^-_j}^{\mathfrak{t}^+_j}\mc R_{\varepsilon_j}(\dot{\widehat{u}}_{\tau_j,\varepsilon_j}(r)) + \mc R^*_{\varepsilon_j}(\overline{w}_{\tau_j,\varepsilon_j}(r))\,\mathrm{d}r\ge \limsup_{j\to +\infty}\int_{\mathfrak{t}^-_j}^{\mathfrak{t}^+_j}p_\V(\dot{\widehat{u}}_{\tau_j,\varepsilon_j}(r),\overline{w}_{\tau_j,\varepsilon_j}(r))\d r\,.
		\end{split}
	\end{equation*}
	Taking into account \eqref{eq:stimapv}, we can continue the above inequality getting
	\begin{align*}
		&\quad\mc E(t,u^-(t))-\mc E(t,u^+(t))\\
		&\ge \limsup_{j\to +\infty}\left[\int_{\mathfrak{t}^-_j}^{\mathfrak{t}^+_j}p_\V(\dot{\widetilde{u}}_{\tau_j,\varepsilon_j}(r),\overline{w}_{\tau_j,\varepsilon_j}(r))\d r-C\int_{\mathfrak{t}^-_j}^{\mathfrak{t}^+_j}(1+\|\overline{w}_{\tau_j,\varepsilon_j}(r))\|_*)\|\dot{\widetilde{u}}_{\tau_j,\varepsilon_j}(r)-\dot{\widehat{u}}_{\tau_j,\varepsilon_j}(r)\|\d r\right].
	\end{align*}
	Thanks to \eqref{eq:wbounds}, \eqref{l2convtris} and the assumption \eqref{eq:vanratio} we easily obtain
	\begin{equation}\label{vanish}
		\int_{\mathfrak{t}^-_j}^{\mathfrak{t}^+_j}(1+\|\overline{w}_{\tau_j,\varepsilon_j}(r))\|_*)\|\dot{\widetilde{u}}_{\tau_j,\varepsilon_j}(r)-\dot{\widehat{u}}_{\tau_j,\varepsilon_j}(r)\|\d r\le C\frac{\tau_j}{\eps_j^2}(\mathfrak{t}^+_j-\mathfrak{t}^-_j)\le C(\mathfrak{t}^+_j-\mathfrak{t}^-_j)\to 0,
	\end{equation}
	and thus, also taking any $p\in RCP_\V$, we get
	\begin{align*}
		&\quad\mc E(t,u^-(t))-\mc E(t,u^+(t))\\
		&\ge \limsup_{j\to +\infty}\int_{\mathfrak{t}^-_j}^{\mathfrak{t}^+_j}p_\V(\dot{\widetilde{u}}_{\tau_j,\varepsilon_j}(r),\overline{w}_{\tau_j,\varepsilon_j}(r))\d r\ge \limsup_{j\to +\infty}\int_{\mathfrak{t}^-_j}^{\mathfrak{t}^+_j}p(\dot{\widetilde{u}}_{\tau_j,\varepsilon_j}(r),\overline{w}_{\tau_j,\varepsilon_j}(r))\d r.
	\end{align*}
	By exploiting property $(iv)$ of Definition~\ref{RCP} and using \eqref{eq:wvanishing} and (d) in Corollary~\ref{prop:inequality1} we deduce
	\begin{equation*}
		\begin{split}
			&\quad\int_{\mathfrak{t}^-_j}^{\mathfrak{t}^+_j}|p(\dot{\widetilde{u}}_{\tau_j,\varepsilon_j}(r),\overline{w}_{\tau_j,\varepsilon_j}(r))-p(\dot{\widetilde{u}}_{\tau_j,\varepsilon_j}(r),\widetilde{w}_{\tau_j,\varepsilon_j}(r))|\d r \\
			&\leq L \max\limits_{r\in [0,T]\backslash\Pi_\tau}\|\overline{w}_{\tau_j,\varepsilon_j}(r)-\widetilde{w}_{\tau_j,\varepsilon_j}(r)\|_* \int_{\mathfrak{t}^-_j}^{\mathfrak{t}^+_j}\|\dot{\widetilde{u}}_{\tau_j,\varepsilon_j}(r)\|\,\mathrm{d}r\\
			&\le L \max\limits_{r\in [0,T]\backslash\Pi_\tau}\|\overline{w}_{\tau_j,\varepsilon_j}(r)-\widetilde{w}_{\tau_j,\varepsilon_j}(r)\|_*\left(\int_{\mathfrak{t}^-_j}^{\mathfrak{t}^+_j}\|\dot{\widetilde{u}}_{\tau_j,\varepsilon_j}(r)-\dot{\widehat{u}}_{\tau_j,\varepsilon_j}(r)\|\d r+\int_{0}^{T}\|\dot{\widehat{u}}_{\tau_j,\varepsilon_j}(r)\|\d r\right)\to0\,.
		\end{split}
	\end{equation*}
	We indeed notice that the first term within the brackets is bounded (actually vanishes) by arguing as in \eqref{vanish}.\\
	Therefore, we finally obtain
	\begin{align*}
		&\quad\mc E(t,u^-(t))-\mc E(t,u^+(t))\ge   \limsup_{j\to +\infty}\int_{\mathfrak{t}^-_j}^{\mathfrak{t}^+_j}p(\dot{\widetilde{u}}_{\tau_j,\varepsilon_j}(r),\widetilde{w}_{\tau_j,\varepsilon_j}(r))\d r\\
		&= \limsup_{j\to +\infty}\int_{\mathfrak{t}^-_j}^{\mathfrak{t}^+_j}p(\dot{\widetilde{u}}_{\tau_j,\varepsilon_j}(r),-\varepsilon_j^2\mathbb{M}\ddot{\widetilde{u}}_{\tau_j,\varepsilon_j}(r)-D_x\mc E(r,\widetilde{u}_{\tau_j,\varepsilon_j}(r)))\d r\,.
	\end{align*}
	The rest of the proof follows closely the argument of Proposition~\ref{proposizione1} from \eqref{tofreeze} on, with $\widetilde{u}_{\tau_j,\varepsilon_j}$ in place of $u^{\varepsilon_j}$, by exploiting \eqref{fatto1} and \eqref{fatto3}. We then omit the details.
	\endproof
	\subsection{An enhanced version of the scheme}
	We conclude by proposing a slightly modified discrete algorithm which allows to get rid of the assumption
	\begin{equation*}
		\frac{\tau}{\eps^2}\quad \text{bounded},
	\end{equation*}
	needed for Proposition~\ref{proposizione1bis}. To describe it, we consider an additional parameter $\delta\in [0,1)$ and in \eqref{scheme} we replace $\eps$ by $\sqrt{\eps^2+\delta}$, and we carefully adjust the initial velocity; namely we consider the following incremental variational scheme:
	\begin{equation}\label{deltalgorithm}
		\begin{cases}
				u^k_{\tau,\varepsilon,\delta}\in \mathop{\rm arg\,min}\limits_{x\in X}\mathcal{F}_{\tau,\sqrt{\varepsilon^2+\delta}}(t^k,x, u^{k-1}_{\tau,\varepsilon,\delta}, u^{k-2}_{\tau,\varepsilon,\delta}),&k\in\mc K_\tau,\\	
				u^0_{\tau,\varepsilon,\delta}:=u_0^\varepsilon\,,\quad u^{-1}_{\tau,\varepsilon,\delta}:= u_0^\eps-\tau \frac{\eps}{\sqrt{\eps^2+\delta}}u_1^\varepsilon\,.
				\end{cases}
	\end{equation}
	For $\delta=0$ we easily recover the original scheme \eqref{schemeincond}.\par 
	Since the only change with respect to previous sections is the replacement of $\eps$ by $\sqrt{\eps^2+\delta}$, all the results still hold true if in the statements one performs the same replacement (without touching the initial data $u_0^\eps$, $u_1^\eps$). In particular part (II) of Theorem~\ref{mainthm} can be rewritten as follows:
	\begin{thm}\label{thmdelta}
			Let $\M,\V$ satisfy \eqref{mass}, \eqref{viscosity} and assume \ref{hyp:E1}--\ref{hyp:E5}, \ref{hyp:E3'}, and \ref{hyp:R1}. Let $u_0^\varepsilon\to u_0$, $\eps u_1^\eps\to 0$. Then for every sequence $(\tau_j,\epsj,\delta_j)\to (0,0,0)$ satisfying
			\begin{equation}\label{delta}
				\sup\limits_{j\in\N}\frac{\tau_j}{\varepsilon_j^2+\delta_j}<+\infty,
			\end{equation} 
			there exists a subsequence (not relabelled) along which the sequence of piecewise affine interpolants $\widehat{u}_{\tau_j,\epsj,\delta_j}$ pointwise converges to an Inertial Virtual Viscosity solution of the rate-independent system \eqref{quasistprob}.
	
		Furthermore, the limit function is an Inertial Balanced Viscosity solution if $\V$ is positive-definite.
	\end{thm}
	The advantage of condition \eqref{delta} is that it is automatically satisfied by any sequence $\delta_j$ for which
	\begin{equation}\label{deltatau}
		\sup\limits_{j\in\N}\frac{\tau_j}{\delta_j}<+\infty,
	\end{equation}
	and thus it permits to separate the vanishing rates of $\tau_j$ and $\epsj$, which can be completely unrelated. \par
	We finally notice that the simplest choice of $\delta=\tau$ in \eqref{deltalgorithm} trivially fulfils \eqref{deltatau} along any subsequence, and thus allows to obtain Theorem~\ref{thmdelta} without really adding a further parameter.
	\bigskip
	
	\noindent\textbf{Acknowledgements.} The authors are members of Gruppo Nazionale per l'Analisi Matematica, la Probabilit\`a e le loro Applicazioni (GNAMPA) of INdAM.
	
	The authors have been supported by the Italian Ministry of Education, University and Research through the Project “Variational methods for stationary and evolution problems with singularities and interfaces” (PRIN 2017).

	{\small
		
		\vspace{10pt} (Filippo Riva) Dipartimento di Matematica “Felice Casorati”, Università degli Studi di Pavia,\\ \textsc{Via Ferrata, 5, 27100, Pavia, Italy}
		\\ 
		\textit{e-mail address}: \textsf{filippo.riva@unipv.it}
		\par	
		
		\vspace{10pt} (Giovanni Scilla) Dipartimento di Scienze di Base ed Applicate per l'Ingegneria, Sapienza Università di Roma, \\
		\textsc{Via Scarpa 16, 00169, Rome, Italy}
		\\
		\textit{e-mail address}: \textsf{giovanni.scilla@uniroma1.it}
		\par
		
		\vspace{10pt} (Francesco Solombrino) Dipartimento di Matematica ed Applicazioni “R. Caccioppoli”, Università degli Studi di Napoli Federico II,\\
		\textsc{Via Cintia, Monte Sant'Angelo, 80126, Naples, Italy}
		\\
		\textit{e-mail address}: \textsf{francesco.solombrino@unina.it}
		\par
		
	}
	
\end{document}